\documentclass[10pt]{amsart}            

\usepackage{amssymb,amsmath,amsthm}
\usepackage{hyperref}
\usepackage{color}
\usepackage{mathrsfs}
\usepackage{amsmath}
\usepackage{amssymb}
\usepackage{amsthm}

\renewcommand\mathcal{\mathscr}

\usepackage{ulem}
\renewcommand{\emph}{\normalem}

\theoremstyle{plain}
\newtheorem{theorem}{Theorem}[section]
\newtheorem*{theorem*}{Theorem}
\newtheorem{lemma}[theorem]{Lemma}
\newtheorem*{lemma*}{Lemma}
\newtheorem{corollary}[theorem]{Corollary}
\newtheorem{proposition}[theorem]{Proposition}

\theoremstyle{remark}
\newtheorem{remark}[theorem]{Remark}
\newtheorem*{remark*}{Remark}

\theoremstyle{definition}
\newtheorem{definition}[theorem]{Definition}
\newtheorem*{definition*}{Definition}

\theoremstyle{notation}

\newtheorem*{notation*}{Notation}

\newtheorem{Standing assumptions}[theorem]{Standing 
assumptions}

\numberwithin{equation}{section}

\newcount\quantno
\everydisplay{\quantno=0}\everycr{\quantno=0}
\newcommand\quant{\advance\quantno by1
                      \ifnum\quantno=1\qquad\else\quad\fi\forall }
\newcommand\itemno[1]{(\romannumeral #1)}

\newcommand\Dom{\mathrm{Dom}}

\newcommand\rest[1]{\kern-.1em
          \lower.5ex\hbox{$\scriptstyle #1$}\kern.05em}

\newcommand\set[1]{{\left\{#1\right\}}}

\renewcommand\mod[1]{\vert{#1}\vert}
\newcommand\bigmod[1]{\bigl\vert{#1}\bigr|}
\newcommand\Bigmod[1]{\Bigl\vert{#1}\Bigr|}

\newcommand\norm[2]{{\Vert{#1}\Vert_{#2}}}

\newcommand\bignorm[2]{\left.{\bigl\Vert{#1}\bigr\Vert_{#2}}\right.}
\newcommand\Bignorm[2]{\left.{\Bigl\Vert{#1}\Bigr\Vert_{#2}}
\right.}

\newcommand\opnorm[2]{|\!|\!| {#1} |\!|\!|_{#2}}
\newcommand\bigopnorm[2]{\bigl|\!\bigl|\!\bigl| {#1} 
\bigr|\!\bigr|\!\bigr|_{#2}}

\newcommand\prodo[2]{\left\langle#1,#2\right\rangle}

\newcommand\wrt{\,\text{\rm d}}

\newcommand\bG{\mathbf{G}}

\newcommand\BC{\mathbb{C}}

\newcommand\BN{\mathbb{N}}
\newcommand\BR{\mathbb{R}}

\newcommand\cB{\mathcal{B}}

\newcommand\cD{\mathcal{D}}

\newcommand\cI{\mathcal{I}}   
 
\newcommand\cJ{\mathcal{J}}

\newcommand\cL{\mathcal{L}}

\newcommand\cR{\mathcal{R}}

\newcommand\cT{\mathcal{T}}   
 
\newcommand\cU{\mathcal{U}}

\newcommand\tr{\mathop{\rm tr}}

\newcommand\al{\alpha}
\newcommand\be{\beta}
    
\newcommand\de{\delta}
  \newcommand\vep{\varepsilon}

\newcommand\la{\lambda}

\newcommand\si{\sigma}

\newcommand\vp{\varphi}

\newcommand\OV{\overline}
\newcommand\funnyk{k\hbox to 0pt{\hss\phantom{g}}}

\newcommand\lu[1]{L^1(#1)}
\newcommand\lp[1]{L^p(#1)}

\newcommand\ld[1]{L^2(#1)}
\newcommand\ldloc[1]{L_{\mathrm{loc}}^2(#1)}

\newcommand\hu[1]{H^1(#1)}
\newcommand\hufin[1]{H_{\mathrm{fin}}^1(#1)}

\newcommand\Xu{X^1}

\newcommand\Xhfin[1]{X_{\mathrm{fin}}^k(#1)}

\newcommand\Xh[1]{X^k(#1)}
\newcommand\Xk{{X^k}}

\newcommand\Xkfin{X_{\mathrm{fin}}^k}

\newcommand\Yh[1]{Y^k(#1)}
\newcommand\Yk{Y^k}

\newcommand\bYone{\mathbb{Y}^1}
\newcommand\bYh[1]{\mathbb{Y}^k(#1)}

\newcommand\bYk{\mathbb{Y}^k}

\newcommand\wh{\widehat}
\newcommand\wt{\widetilde}
\newcommand\whH{\widehat{\phantom{G}}\hbox to 0pt{\hss $H$}}

\newcommand\emspace{\hbox to 6pt{\hss}}
\newcommand\ds{\displaystyle}

\newcommand\rmi{\hbox{\rm (i)}}
\newcommand\rmii{\hbox{\rm (ii)}}
\newcommand\rmiii{\hbox{\rm (iii)}}
\newcommand\rmiv{\hbox{\rm (iv)}}
\newcommand\rmv{\hbox{\rm (v)}}

\newcommand\One{{\mathbf{1}}}

\newcommand\e{\mathrm{e}}

\newcommand\Inj{\mathrm{Inj}}

\newcommand\Ric{\mathop{\rm Ric}}

\newcommand\QBk{q^2_k(B)}
\newcommand\QOVBk{q^2_k(\OV B)}
\newcommand\Qk[1]{q_k^2({#1})}
\newcommand\QBkperp{q^2_k(B)^{\perp}}

\newcommand\Qkperp[1]{q_k({#1})^{\perp}}
\newcommand\PBk{\pi_{B,k}}

\newcommand\bktwo{h_k^2}
\newcommand\Pik[1]{\pi_{#1,k}}

\newcommand\OVB{\overline{B}}

\DeclareSymbolFont{EUEX}{U}{euex}{m}{n}

\DeclareSymbolFont{euexlargesymbols}{U}{euex}{m}{n}
\DeclareMathSymbol{\intop}{\mathop}{euexlargesymbols}{"52}
     \def\int{\intop\nolimits}

\DeclareSymbolFont{euexsymbols}     {U}{euex}{m}{n}
\DeclareMathSymbol{\smallint}{\mathop}{euexsymbols}{"52}

\begin{document}

\title[The dual of Hardy type spaces]
{Harmonic Bergman spaces, the Poisson equation\\
and the dual of Hardy-type spaces \\
on certain noncompact manifolds}

\subjclass[2000]{} 


\thanks{Work partially supported by PRIN 2010 ``Real and 
complex manifolds: 
geometry, topology and harmonic analysis".}

\author[G. Mauceri, S. Meda and M. Vallarino]
{Giancarlo Mauceri, Stefano Meda and Maria Vallarino}

\address{Giancarlo Mauceri: Dipartimento di Matematica\\ 
Universit\`a di Genova\\
via Dodecaneso 35\\ 16146 Genova\\ Italy 
-- mauceri@dima.unige.it}

\address{Stefano Meda: 
Dipartimento di Matematica e Applicazioni
\\ Universit\`a di Milano-Bicocca\\
via R.~Cozzi 53\\ I-20125 Milano\\ Italy
-- stefano.meda@unimib.it}

\address{Maria Vallarino:
Dipartimento di Scienze Matematiche\\
Politecnico di Torino\\
Corso Duca degli Abruzzi, 24\\ 
I-10129 Torino \\
Italy
-- maria.vallarino@polito.it}

\begin{abstract}
In this paper we consider a complete connected noncompact 
Riemannian manifold $M$ with bounded geometry and spectral 
gap.  
We realize the dual space $\Yh{M}$ of the Hardy-type 
space $\Xh{M}$, introduced in a previous paper of the authors,
as the class of all locally square integrable functions satisfying 
suitable $BMO$-like conditions, where the role of the constants is
played by the space of global $k$-quasi-harmonic functions.
Furthermore we prove that $\Yh{M}$ is also 
the dual of the space $\Xkfin(M)$ of finite linear combination of $
\Xk$-atoms.   
As a consequence, if $Z$ is a Banach space and $T$ is
a $Z$-valued linear operator defined on $\Xkfin(M)$, then $T$ 
extends to a 
bounded operator from $\Xk(M)$ to $Z$ if and only if it is uniformly
bounded on $\Xk$-atoms.
To obtain these results we prove the global solvability of 
the generalized Poisson equation $\cL^ku=f$ with $f\in 
L^2_{\mathrm{loc}}(M)$ and we study some properties of 
generalized Bergman spaces of harmonic 
functions on geodesic balls.
\end{abstract}
\par
\maketitle

\setcounter{section}{0}
\section{Introduction} \label{s:Introduction}

A seminal result of C.~Fefferman \cite[Thm~2]{FeS} identifies the 
Banach dual
of the Hardy space $\hu{\BR^n}$ as $BMO(\BR^n)/\BC$, the 
space
of all functions of bounded mean oscillation modulo constants.
Functions in $BMO(\BR^n)$ possess the nice property of being
``well approximated'' on each ball by constants, to wit
\begin{equation} \label{f: intro uno}
\sup_{B} \inf_{c\in \BC} \frac{1}{\mod{B}} \int_{B} \mod{f(x) - c}^2 
\wrt x
< \infty.
\end{equation}
The continuous linear functionals on $\hu{\BR^n}$
are precisely those, which, restricted to 
finite linear combinations $f$ of $H^1$-atoms, are of the form
\begin{equation} \label{f: intro due}
\la_g(f)
:= \int_{\BR^n} f(x) \, g(x) \wrt x
\end{equation}
for some function $g$ in $BMO(\BR^n)$. 
\par  
Recently several authors have investigated Hardy spaces on non-compact doubling Riemannian manifolds \cite{AMR, MRu, Ru}. On non-doubling manifolds, versions of the so called local Hardy spaces, introduced by D. Goldberg in the context of $\BR^n$  \cite{Gol}, have been investigated in \cite{Tay, CMcM}. These local Hardy spaces are well adapted to obtain endpoint estimates for singular integrals whose kernels have only local singularities, such as, for instance, pseudodifferential operators. However, it is known that singular integral operators whose kernel is also singular at infinity, such as  Riesz transforms or imaginary powers of the Laplacian, do not map the local Hardy spaces in $L^1(M)$.
To overcome this problem, in \cite{MMV1, MMV2} the authors introduced and studied the properties of a family of global Hardy-type spaces  on a class of non-doubling manifolds that includes all symmetric spaces of 
the noncompact type. 
\par
In this paper we aim at proving a version of Fefferman's result for 
this new class of spaces.
A striking difference between the aforementioned classical result
and our version thereof is that the role played by constants in the 
former
will be played in the latter by quasi-harmonic functions, i.e., 
solutions to the
(generalised) Poisson equation $\cL^k u = c$ for some positive 
integer $k$ and
constant $c$.  Here $\cL$ denotes minus the Laplace--Beltrami 
operator on $M$.
\par
We elaborate on this.
In \cite{MMV1, MMV2} we defined a strictly decreasing
sequence $X^1(M), X^2(M), X^3(M),...$ of subspaces of $\lu{M}$, 
where $M$
is a complete connected noncompact 
Riemannian manifolds with Ricci curvature bounded from below, 
positive injectivity radius and spectral gap (see also \cite{Vo} for 
an
interesting variant of the spaces $\Xh{M}$).  Note that these 
manifolds 
are of exponential volume growth, 
hence their Riemannian measure $\mu$ is \textit{nondoubling}.
Important examples of manifolds with these properties are 
nonamenable
connected unimodular Lie groups equipped with a left
invariant Riemannian distance, and symmetric spaces
of the noncompact type with the Killing metric. 
The spaces $\Xk(M)$ share with 
the classical Hardy space $\hu{\BR^n}$ the following properties 
(see
\cite{MMV1,MMV2}):
\begin{itemize}
\item[(i)] 
if $p$ is in $(1,2)$, then the Lebesgue space $\lp{M}$ is an 
interpolation 
space between $\Xk(M)$ and $\ld{M}$;
\item[(ii)]  some interesting operators, such as the Riesz 
transforms associated 
to $\cL$ and the purely imaginary powers of $\cL$,
are bounded from $\Xk(M)$ to $\lu{M}$ for $k$ large enough (improvements thereof will appear in \cite{MMV3}); 
\item[(iii)]
the space $\Xu(M)$ admits an atomic decomposition in terms of 
atoms, 
which are defined much as in the classical case, 
but are supported in balls of radius at most one
and satisfy an appropriate \textit{infinite dimensional} cancellation 
condition.    
Under the additional assumption that some (depending on $k$) 
covariant derivatives
of the Ricci tensor are bounded, the same holds for $\Xk(M)$, $k
\geq 2$.
\end{itemize}
For this reason we call the spaces $\Xk(M)$ \textit{generalised 
Hardy spaces}.  
They play for harmonic analysis on $M$ a role similar to that played by
$\hu{\BR^n}$ on $\BR^n$ \cite{St2} and, more generally, by
the Coifman--Weiss Hardy space \cite{CW}  on spaces of 
homogeneous type. \par
In order to describe the cancellation condition alluded to in \rmiii\ 
above, 
we define, for each geodesic ball~$B$, 
the class $\Qk{\OV B}$ of all {\it $k$-quasi-harmonic functions} on $\OV B$ as the class of functions  that are restrictions to $\OV{B}$ of functions $v$ such that $$\cL^k v=const$$ in some open neighbourhood of $\OV{B}$.
\par
Atoms in $\Xk(M)$ are then $\ld{M}$ functions $A$ 
with support contained in a ball $B$
of radius at most one such that 
\begin{itemize}
\item[\rmi] $\int A\,v\wrt\mu= 0\qquad \forall v\in \Qk{\OV B};$
\item[\rmii] $\norm{A}{2}\le \mu(B)^{-1/2}.$
\end{itemize}

The space $\Xk(M)$ is the space of all (possibly infinite) linear
combinations of $\Xk$-atoms with $\ell^1$ coefficients, endowed 
with the standard 
``atomic norm''.  
If we consider atoms with support contained in balls  
of radius at most $s>0$, instead of atoms with support in balls
of radius at most $1$, we obtain the same space of functions, and 
the two 
corresponding ``atomic norms'' are equivalent \cite{MMV2}.  
In view of this observation, we may choose the ``scale parameter'' 
$s$ 
equal to $s_0 :=(1/2)\, \Inj(M)$.  This
will simplify some of the arguments below, and avoid many 
annoying
technicalities. We shall call atoms supported in balls od radius at most $s_0$ {\it admissible}.

An important point that we overlooked in \cite{MMV2}, and that we 
shall
discuss in Section~\ref{s: Background material}, is that the 
cancellation condition
for $\Xk$-atoms may be equivalently formulated, at least for atoms 
$A$
with support in small balls, by requiring that
$A$ is orthogonal to $q_k(M)$, the space of all global $k$-quasi-
harmonic functions, 
i.e. the space of all solutions to the (generalised) Poisson equation
$$
\cL^k u = c,
$$
where $c$ is an arbitrary constant.  In Section~\ref{s: Solvability} 
we shall prove that the generalised Poisson equation has global 
solutions and
that if $r_B$ is small enough, then 
functions in $\Qk{\OV B}$ may be approximated in the $\ld{B}$ 
norm to
any degree of precision by global $k$-quasi-harmonic functions. 
This suggests to define the \textit{generalised} $BMO$ space 
$GBMO^k(M)$ as the space of all locally square integrable 
functions
$G$ on $M$ such that 
\begin{equation} \label{f: intro tre}
\norm{F}{GBMO^k}
:= \sup_B \inf_{V \in q_k(M)}  
\Bigl(\frac{1}{\mu(B)} \int_{B} \bigmod{G - V}^2 \wrt \mu\Bigr)^{1/2}
< \infty,
\end{equation}
where the supremum is taken with respect to all balls of radius at 
most $s_0$.  
Note that this ``norm'' annihilates all global $k$-quasi-harmonic 
functions,
and defines a genuine norm on the quotient space $GBMO^k(M)/
q_k(M)$.  
Loosely speaking, functions in $GBMO^k(M)$ may be ``well 
approximated''
on balls of radius at most $s_0$ by global $k$-quasi-harmonic 
functions.
Our main result, Theorem~\ref{t: main}, 
states that the Banach dual of $\Xk(M)$ is isomorphic to 
$GBMO^k(M)/q_k(M)$.  
Specifically, the continuous linear functionals on $\Xk(M)$ are 
precisely
those, which, restricted to finite linear combinations $F$ of $\Xk$-
atoms,  
are of the form
\begin{equation} \label{f: intro quattro}
\la_G(F)
:= \int_M F \, G \wrt \mu
\end{equation}
for some function $G$ in $GBMO^k(M)$.   
Note the analogy between the classical case \eqref{f: intro uno}, 
\eqref{f: intro due}, and our setting \eqref{f: intro tre}, 
\eqref{f: intro quattro}.  
It is an interesting problem
to determine explicitly the function $G$ that corresponds to a 
given functional 
$\la$.  We solve this problem in Section~\ref{s: dual of H-t s}.  It 
may be worth
observing that one of the steps in the proof is showing that
the solutions of the generalised Poisson equation 
$$
\cL^k u = g
$$
with datum $g$ in $BMO(M)$ are in $GBMO^k(M)$, with control of 
the norms. 
As a consequence of our analysis, we prove in Section~\ref{s: 
dual of Xkfin} that 
the $\Xh{M}$-norm and the norm
$$
\inf \Bigl\{ \sum_{j} \mod{c_j}:  F = \sum_j c_j \, A_j \Bigr\},
$$
where the infimum is taken over all representations
of $F$ as a \textit{finite} sum of admissible $X^k$-atoms,
are equivalent on the space $\Xkfin(M)$ of finite linear 
combination
of $X^k$-atoms.  This implies that if $Z$ is a Banach
space and $\cT$ is a $Z$-valued linear operator on finite 
combinations of $X^k$-atoms
that is uniformly bounded on admissible $X^k$-atoms, then it 
extends 
to a bounded linear operator from $\Xh{M}$ to $Z$.  Thus, the 
atomic
decomposition is really useful to test the boundedness
of linear operators defined on finite linear combinations of $X^k$-
atoms. 
See, on this delicate point, \cite{MSV1,MSV2,MM} and the 
references therein. 
This result has already been implicitly used in \cite{MMV1,MMV2}, 
where, in order 
to show that certain singular integral operators are bounded from 
$\Xk(M)$
to $\lu{M}$, we simply checked that they are uniformly bounded on 
$X^k$-atoms.  

Further applications of the theory developed in this paper
to the boundedness of spectral multipliers
of $\cL$ and Riesz transforms will appear in \cite{MMV4}.

We briefly outline the content of this paper.
Section~\ref{s: Solvability} is devoted to the study of
the solvability of the generalised Poisson equation. 
In Section~\ref{s: Quasi} we introduce various classes
of $k$-quasi-harmonic functions and study their mutual relations. 
In Section~\ref{s: Background material}, after stating the basic 
geometric assumptions  on the manifold $M$ and 
their analytic consequences, we recall the definition 
of the spaces $\hu{M}$, $\Xh{M}$, $\Yh{M}$ and their main
properties. 
Our main result is proved in Section~\ref{s: dual of H-t s}. 
Finally, in Section~\ref{s: dual of Xkfin} we shall prove
that the dual of $\Xhfin{M}$ is isomorphic to that of $\Xh{M}$, and 
draw some
consequences concerning the extendability of Banach-valued
linear operators uniformly bounded on $X^k$-atoms. 
 
We shall use the ``variable constant convention'', and denote by $C,
$
possibly with sub- or superscripts, a constant that may vary from 
place to 
place and may depend on any factor quantified (implicitly or 
explicitly) 
before its occurrence, but not on factors quantified afterwards. Throughout the paper $\One_E$ will denote the indicator function of the set $E$.

\section{Solvability of the Poisson equation}
\label{s: Solvability}

In this section $M$ will denote a connected complete 
$n$-dimensional Riemannian manifold
of infinite volume with Riemannian measure $\mu$.
We assume that the bottom $b$ of the $\ld{M}$ spectrum of $\cL$ 
is 
strictly positive.   
The aim of this section is to investigate the solvability of the 
(generalised) Poisson
equation 
$$
\cL^k u = g,
$$
where $g$ is a datum in $\ldloc{M}$ and $k$ a positive integer.  
Clearly, if $U$ is a distributional solution of this equation, 
any other solution is of the form $U+H$, where $H$ solves
the generalised Laplace equation $\cL^k H = 0$, i.e.
$H$ is a global $k$-harmonic function on $M$, according to the 
terminology 
that we shall introduce in Definition~\ref{def: quasi-harmonic}.   

The proof of the solvability hinges on the following approximation 
result of 
$k$-harmonic functions on certain compact subsets of $M$ by global 
$k$-harmonic functions 
on $M$. 
The proof for $k=1$ can be found in \cite[Thm~3.10]{BB}; the case $k>1$ is a straightforward adaptation of the argument given there.
We recall that if $K$ is a closed subset of $M$, 
then a {\it hole} of $K$ is any component of $M\setminus K$ which 
is bounded.

\begin{lemma}[\textbf{Walsch--Pfluger--Lax--Malgrange}] \label{l: 
BB} 
Let $K$ be a compact subset of $M$  without holes and $k$ a positive integer. If $v$ is a 
solution of the equation $\cL^k v=0$ in a neighbourhood of $K$ and
$\vep>0$, then there is a function $u$ such that $\cL^k u=0$ in 
$M$ and $\sup_K\mod{v-u}<\vep$.
\end{lemma}

\begin{theorem} \label{t: sol} 
Suppose that $M$ is a complete, noncompact, 
Riemannian manifold
with spectral gap and that $k$ is a positive integer. 
Then for every $f$ in $\ldloc{M}$ there exists $u$ in $\ldloc{M}$ 
such that $\cL^ku=f$ in the sense of distributions.
\end{theorem}

\begin{proof}  
First we consider the case $k=1$.  Fix a reference point $o$ in $M
$ and denote 
by $B_R$ the open ball of radius $R$ and centre $o$, and by $
\wh B_R$ 
the union of $B_R$ with the bounded connected components of 
$M\setminus B_R$. 
Then  $\wh{B}_R$ has no holes and $M$ is the union of the 
increasing 
sequence of bounded open sets $\{\wh {B}_R\}$.  

Fix $\vep>0$ small.  
Set $v_1=\cL^{-1}(f\,\One_{\wh{B}_{1+\vep}})$. This makes sense, 
because $\cL^{-1}$ is
bounded on $\ld{M}$.  Then $v_1$ is in 
$\ld{M}$ and solves the equation $\cL v_1=f$ in $\wh{B}_{1+\vep}
$. If 
$w\in \ld{M}$ is a solution of $\cL w=f$ in $\wh{B}_{2+\vep}$ (for 
instance 
$w=\cL^{-1}(f\, \One_{\wh{B}_{2+\vep}})$), then $\cL (v_1-w)=0$ 
in $\wh{B}_{1+\vep}$. Hence, by Lemma \ref{l: BB}, there exists 
$h$ such that 
$\cL h=0$ in $M$ and $\sup_{\wh{B}_1}\mod{v_1-w-h}<1/2$. 
Thus, setting 
$v_2=w+h$, one has $\cL v_2=f$ in $\wh{B}_{2+\vep}$ 
and $\sup_{\wh{B} _1}\mod {v_1-v_2}<1/2$. 
Iterating this argument, one constructs a sequence  
of functions $v_j$ in $\ld{M}$ such that $\cL v_j=f$ in  $\wh{B}_{j+
\vep}$ and 
$\sup_{\wh{B}_j}\mod{v_j-v_{j+1}}<2^{-j}$. 
Thus $v_j$ converges in $\ldloc{M}$, whence in the sense of 
distributions,
to a limit $u$, which satisfies $\cL u=f$ in $M$.

The case $k>1$ can be reduced to $k=1$, by observing that the 
equation $\cL ^ku=f$ is equivalent to the system of $k$ equations 
$\cL  
u_\ell=u_{\ell-1}$, $\ell=1,\ldots,k$, where $u_0=f$.
\end{proof}

\section{Quasi-harmonic functions and Bergman spaces}
\label{s: Quasi}   

We introduce various spaces of functions on $M$ that will play
an important role in what follows and investigate their mutual 
relations.
Here $M$ is as in Section~\ref{s: Solvability}.  
Recall that $\cL$ is an elliptic operator.  Thus, given an open 
subset $\Omega$ of $M$,
a positive integer $k$ and a constant $c$, 
every solution $u$ of the equation 
$$
\cL^k u = c \, \One_\Omega
$$
is smooth in $\Omega$.
\par
The operator $\cL$ has been defined in the introduction as the unique self-adjoint extension of minus the Laplace-Beltrami operator acting on $C^\infty_c(M)$. We recall that  
the domain of $\cL$ in $L^2(M)$ is the space 
$
\Dom(\cL)=\set{u\in L^2(M): \cL u\in L^2(M)},
$ where $\cL u$ is interpreted in the sense 
of distributions \cite{Str}.
Henceforth we shall also denote by $\cL$ the natural extension of the Laplace-Beltrami operator to distributions.
\begin{definition} \label{def: quasi-harmonic}
Suppose that $k$ is a positive integer, and that
$\Omega$ is a \textsl{bounded} open subset of $M$. We say that a function 
$v:\Omega\to \BC$ 
is \textit{$k$-quasi-harmonic} on $\Omega$ if 
$\cL^k v$ is constant on $\Omega$ (in the sense of distributions, hence in the classical sense, since $v$ is smooth by elliptic regularity). We shall denote by  $\Qk{\Omega}$ the 
space of  
$k$-quasi-harmonic functions on $\Omega$ which belong to $L^2(\Omega)$.  
The subspace of $\Qk{\Omega}$ of all functions such that $\cL^k v=0$ in $\Omega$ will be denoted by $h^2_k(\Omega)$ and will be called the {\it ($k^{th}$ generalised) Bergman space} on $\Omega$. 
\par
Suppose now that $K$ is a compact subset of $M$. We say that $w:K\to\BC$ is {\it $k$-quasi-harmonic} on $K$ if $w$ is the restriction to $K$ of a function in $\Qk{\Omega}$, for some open set $\Omega$ containing $K$. We shall denote by $\Qk{K}$ the space of all $k$-quasi-harmonic functions on $K$. 
The subspace of $\Qk{K}$ of all functions which are restrictions to $K$ of functions in $h^2_k(\Omega)$ will be denoted by $h^2_k(K)$ and will be called the {\it ($k^{th}$ generalised) Bergman space} on $K$. 
\par
Finally we shall denote by $q_k(M)$ the space of all $k$-quasi-harmonic functions on $M$. Notice  that $q_k(M)$ is a space of functions in $C^\infty(M)$. 
\end{definition}

Clearly $\bktwo(\Omega)$ is a subspace of $\Qk{\Omega}$ of codimension 
one.  
Indeed, we have the vector space decomposition
$$
\Qk{\Omega}
= \bktwo(\Omega) \oplus \BC \,(\cL^{-k} \One_{\Omega}){\big|}_\Omega.
$$
Note that $\cL^{-k} \One_\Omega$ is in $\ld{M}$, for the bottom $b$ of 
the spectrum of 
$\cL$ is assumed to be positive, whence $\cL^{-k}$ is bounded on 
$\ld{M}$.
\par
Observe that both $\Qk{\Omega}$ and $h^2_k(\Omega)$ are closed subspaces of $L^2(\Omega)$. Indeed, in view of the decomposition above
it suffices to prove that $\bktwo(\Omega)$ is closed.  Now, 
if  $\{v_n\}$  is a sequence in $\bktwo(\Omega)$ that converges to 
$v$  in $\ld{\Omega}$, then $\cL^k v_n$ tends to $\cL^k v$ in the sense of distributions. Thus $\cL v=0$,  whence $v$ is 
in $\bktwo(\Omega)$.  
Clearly $\Qk{\OV \Omega}$ is contained in $\Qk{\Omega}$.  
We shall prove below that if the boundary of $\Omega$ is smooth 
then 
$\Qk{\OV \Omega}$ is dense in $\Qk{\Omega}$.  To prove this, we need
a few preliminary facts.

\begin{definition} \label{def: Sobolev}
For a positive integer $m$ denote by $H^m(M)$ the Sobolev 
space of
order $m$, i.e., the completion of 
$$
\set{u\in C^\infty(M): \nabla^j u\in L^2(M), \, j=0,1\ldots,m}
$$
with respect to the norm
$$
\norm{u}{H^m}
:= \Bigl(\sum_{j=0}^m \bignorm{\nabla^j u}{2}^{\!\!\!\! 2}\Bigr)^{1/2}.
$$
\end{definition}

\noindent
See \cite{He} and the references therein for more on Sobolev
spaces on manifolds.  

Given a compact subset $K$ of $M$, denote 
by $H^m(M)_K$ the subspace of $H^m(M)$ of all functions whose 
support is contained 
in  $K$, by $\mathring K$ the interior of $K$, 
and by $H^m_0(\mathring K)$ the closure of $C^\infty_c(\mathring 
K)$ in $H^m(M)$.

Suppose that $u$ is a function in $\Dom(\cL^k)$ that vanishes in 
the complement of $K$.  Then $\cL^ku$ is in $\ld{M}$
and vanishes in $K^c$.  By identifying $\cL^k u$ with its restriction
to $K$, we may interpret $\cL^k$ as a map from $\Dom(\cL^k)_K$ into $\ld{K}$.
We shall make this identification in the sequel without further comment. Henceforth, if $S$ is a subspace of $L^2(B)$, we shall denote by $S^\perp$ its orthogonal in $L^2(B)$.
\begin{lemma}\label{l: iso}
Let $K$ be a compact subset of $M$.  The following hold:
\begin{enumerate}
\item[\itemno1]
$H^{2k}(M)$ is contained in $\Dom(\cL^k)$;
\item[\itemno2]
the map $\cL^k$ is a Banach space isomorphism 
between $H^{2k}(M)_K$ and $\bktwo(K)^\perp$ (the orthogonal complement of $h^2_k(K)$ in $L^2(K)$). 
\end{enumerate}
\end{lemma}
\begin{proof}
First we prove \rmi. Let $T^{k}M$ be the bundle of covariant tensors of order $k$ and denote by $\tr: T^{k+2}M\to T^kM$ a trace, i.e. a metric contraction. Then for all sections $T$ of $T^{k+2}M$
$$
\mod{\tr(T)(x)}_x \le\,\sqrt{n}\, \mod{T(x)}_x \qquad\forall x\in M.
$$
as can be easily seen by computing the trace in  local coordinates  given by an orthonormal frame and applying Schwarz's inequality.

Next we observe that the Laplacian $\cL$ is bounded from the Sobolev space $H^{2k+2}(M)$ to $H^k(M)$. Indeed if $u\in H^{2k+2}(M)$ then, since the trace commutes with covariant derivatives, 
$$
\nabla^j \cL u=\nabla^j \tr(\nabla^2 u)=\tr (\nabla^{j+2}u) \qquad\forall j=0,1,\ldots,2k.
$$
Thus
$$
\norm{\nabla^j u}{2}\le \sqrt{n} \norm{\nabla^{j+2}u}{2},
$$
whence the boundedness of $\cL$ from $H^{2k+2}(M)$ to $H^k(M)$ follows.\par
To prove that $H^{2k}(M)\subset \Dom(\cL^k)$ 
we consider first the case $k=1$.
If $u\in H^2(M)$ then
$$
\norm{\cL u}{2}=\norm{\tr \nabla^2 u}{2}\le \sqrt{n} \norm{\nabla^2 u}{2}. 
$$
The inclusion $H^2(M)\subset \Dom(\cL)$ follows, since $C^\infty(M)\cap H^2(M)$ is dense in $H^2(M)$ and $\cL$ is closed. 
Finally, 
since $\Dom(\cL^k)=\set{u\in \Dom(\cL): \cL u\in \Dom(\cL^{k-1})}$, the inclusion $H^{2k}(M)\subset \Dom(\cL^k)$ follows by induction on $k$.

Now we prove \rmii.
First we show that $\cL^k$ maps $H^{2k}(M)_K$ into $\bktwo(K)^\perp$. 
Suppose that $u$ is in $H^{2k}(M)_K$ and $v$ is in $\bktwo(K)$. 
Denote by $\wt v$ a smooth function with compact support 
which is $k$-harmonic in an open neighbourhood of $K$ and satisfies
$\wt v_{\vert_K} = v$.  Then 
$$
\begin{aligned}
\int_{K} v\, \cL^k u  \wrt \mu
= \int_{M} \wt v\, \, \cL^k u  \wrt \mu   = \int_{M} \cL^k \wt v\,\,  u  \wrt \mu   = 0
\end{aligned}
$$
because the support of $u$ is contained in $K$ and $\cL^k v$ 
vanishes in a neighbourhood of~$K$.

Since the bottom of the $L^2$-spectrum of $\cL$ is strictly 
positive, $\cL^k$
is injective on its domain, hence on $H^{2k}(M)_K$, for this is a 
subspace of $\Dom(\cL^k)$ by \rmi\ above.

Next we prove that $\cL^k$ is onto.  Suppose that $v$ is in $\bktwo(K)^\perp$.  Denote by $\tilde{v}$ the extension of $v$ to a function on $M$ that vanishes off $K$.
Set $u: = \cL^{-k}\tilde{v}$. Clearly $u$ belongs to $\Dom(\cL^k)$.  We 
shall prove that
$u$ is in $H^{2k}(M)_K$.  First we show that the support of $u$ is 
contained in $K$.  
For every smooth function $\phi$ with support contained in $K^c$
\begin{align*}
\int_M \phi\, \, u  \wrt\mu  
&= \int_M \cL^k\cL^{-k} \phi\, \, u \wrt \mu \\
 &= \int_M  \cL^{-k} \phi\, \, \cL^k u \wrt \mu  
 = \int_M \cL^{-k} \phi \, \, \tilde{v}\wrt \mu  = 0;
\end{align*}
the last equality follows from the fact that 
$\cL^{-k}\phi$ is $k$-harmonic in a neighbourhood of $K$ (hence
its restriction to $K$ belongs to $\bktwo(K)$) and $v$ is in $\bktwo(K)^\perp$.

Since $\cL^k$ is an elliptic operator of order $2k$ and both $u$ 
and $\cL^k u$ 
are functions in $\ld{M}$ with compact support, $u$ is in $H^{2k}
(M)$.  
Thus, $\cL^{2k}$ maps $H^{2k}(M)_K$
onto $\bktwo(K)^\perp$ in a one-to-one fashion.  Furthermore $
\cL^k$ is a continuous
operator from $H^{2k}(M)_K$ to $\bktwo(K)^\perp$. The closed 
graph theorem 
then implies that $\cL^{-k}$ is continuous, thereby concluding the 
proof that
$\cL^k$ is a Banach space isomorphism.  
\end{proof}

\begin{theorem} \label{t: bergman-sobolev}
Let $K$ be a compact subset of $M$.  The following are 
equivalent:
\begin{enumerate}
\item[\itemno1]
$\bktwo(K)^\perp=\bktwo(\mathring K)^\perp$ 
\item[\itemno2]
$H^{2k} (M)_K=H^{2k}_0(\mathring K)$.
\end{enumerate}
\end{theorem}

\begin{proof}
First we prove that \rmi\ implies \rmii.  
Clearly $H^{2k}_0(\mathring K)\subseteq H^{2k}(M)_K$, so that it 
suffices to prove
the inclusion $H^{2k}(M)_K\subseteq H^{2k}_0(\mathring K)$, 
equivalently that $C^\infty_c(\mathring K)$ is dense in  $H^{2k}(M)_K$. 
By Lemma~\ref{l: iso}, this is equivalent to the density of
$\cL^k(C^\infty_c(\mathring K))$ in 
$\bktwo(K)^\perp=\bktwo(\mathring K)^\perp$, i.e., that 
the orthogonal space to $\cL^k(C^\infty_c(\mathring K))$ in 
$\bktwo(\mathring K)^\perp$ is the null space.

Suppose that $f$ is a function in $\bktwo(\mathring K)^\perp$ that 
is orthogonal to $\cL^k \phi$, for every $\phi$ in $C^\infty_c(\mathring K)$. 
Denote by $\wt f$ the extension of $f$ to a function on $M$ 
which vanishes in $(\mathring K)^c$, and consider the distribution $\cL^k \wt f$. 
Then
$$
\prodo{\phi}{\cL^k \wt f}
= \prodo{\cL^k \phi}{\wt f} 
= 0 
\quant \phi\in C^\infty_c(\mathring K).
$$
Thus, $\cL^k \wt f=0$ in $\mathring K$, so that $f$ belongs to $
\bktwo(\mathring K)$.
But this implies that $f=0$ for $f$ is in  
$\bktwo(\mathring K)^\perp\cap \bktwo(\mathring K)$.

Next we prove that \rmii\ implies \rmi. 
Observe that the obvious inclusion $\bktwo(K)\subseteq 
\bktwo(\mathring K)$
implies the containment
$\bktwo(\mathring K)^\perp\subseteq \bktwo(K)^\perp$.
Thus, it suffices to show that the assumption  
$H^{2k} (M)_K=H^{2k}_0(\mathring K)$ implies
$$
\bktwo(K)^\perp\subseteq \bktwo(\mathring K)^\perp.
$$ 
Suppose that $v$ is in $\bktwo(K)^\perp$.  Write $\wt v$ for the 
extension of $v$ to $M$ which vanishes in $K^c$. By Lemma~\ref{l: iso} 
there exists $u$ in $H^{2k}(M)_K$ such that $\cL^k u = \wt v$. 
The assumption $H^{2k}(M)_K=H^{2k}_0(\mathring K)$ implies the existence of a 
sequence $\{\phi_n\}$ of functions in $C^\infty_c(\mathring K)$ that is 
convergent to $u$ in $H^{2k}(M)$.  Then, if for every $f$ in $\bktwo(\mathring K)$ we denote by $\tilde{f}$ its extension to $M$ that vanishes off $K$,
\begin{align*}
\int_{\mathring K} v\, \, f  \wrt \mu 
 &= \int_M \wt v\, \, \wt f \wrt \mu 
 = \int_M \cL^k u\, \, \wt f \wrt \mu \\ 
& = \lim_{n\to \infty} \int_M \cL^k \phi_n\, \, \wt f \wrt \mu 
 = \lim_{n\to \infty} \prodo{\phi_n}{\cL^k \wt f},
\end{align*} 
which vanishes because the support of $\phi_n$ is contained in 
$\mathring K$ and $\cL^k \wt f=0$ in $\mathring K$. 
Therefore $v$ is in $\bktwo(\mathring K)^\perp$, as required.
\end{proof}

\noindent
The result above raises the following question: which compact 
subsets $K$ of $M$
satisfy the requirement $H^{2k} (M)_K=H^{2k}_0(\mathring K)$? 
In the case where $K=\overline{\mathring K}$ is a domain whose 
boundary
is a smooth $(n-1)$-dimensional manifold, the spaces $H^{2k}
(M)_K$ and $H^{2k}_0(\mathring K)$ coincide by a well known 
result of
Lions and Magenes \cite[Theorem~2, p. 259]{E}.  More generally, 
we may use a version
of the segment condition for manifolds (see \cite[Theorem~5.29, 
p.~159]{AF}).  
Unfortunately, this is 
not very useful in the setting of Riemannian manifolds, for the 
boundary of
geodesic balls may contain even cusps (think of the elementary 
example of
a cilynder in $\BR^3$).  However, it is a classical fact that if $r < 
\Inj_p$,
then the boundary of $B(p,r)$ is a smooth $(n-1)$-dimensional 
submanifold of $M$,
a fact that will be used without further comment in the sequel.  

\noindent
Note that for every nonnegative integer $k$ and every
open ball $B$, we have the orthogonal decompositions
$$
\ld{B}
=  \QBkperp + {\QBk}
=  \Qkperp{\OVB} + \OV{\Qk{\OVB}}.
$$
In fact, these decompositions coincide, at least for all $B$ such 
that
$r_B < \Inj_{c_B}$, as the following result shows.

\begin{proposition} \label{p: structure CHhperp}
Suppose that $B$ is an open ball in $M$, that $r_B < \Inj_{c_B}$ 
and 
that~$k$ is a positive integer.  The following hold:
\begin{enumerate}
\item[\itemno1]
$
\OV{\Qk{\OV B}}
= \Qk{B};
$
\item[\itemno2]
$\Qk{B}$ is the set of all $v$ in $\ld{B}$ such that there exists a
sequence $\{v_n\}$ of global $k$-quasi-harmonic functions such 
that 
$$
\lim_{n\to\infty} \int_B \bigmod{v-v_n}^2 \wrt \mu = 0.
$$  
\end{enumerate}
\end{proposition}

\begin{proof}
To prove \rmi\ we first prove that if $v$ is in $\OV{\Qk{\OV B}}$, 
then $v$ is smooth
on $B$ and $\cL^k v$ is constant therein, i.e., $v$ is in $\Qk{B}$.
Indeed, there exists a sequence $\{v_j\}$ of functions in $C^\infty_c(M)$, such that 
${\cL^k}v_j$ is constant in a neighborhood of $\overline{B}$, that
converges to $v$ in $\ld{{B}}$. Then $\{{\cL^k}v_j\}$ converges to $\cL^k v$ in $\cD'(B)$
so that $\cL^k v$ is constant on $B$ in the sense of distributions,
and, by elliptic regularity, $v$ is smooth on $B$,  
as required.
\par
Conversely, suppose that $v$ is in $\ld{B}$ and that $\cL^k v = c$ 
on $B$ in the sense of distributions for some constant $c$. 
Then $v$ is smooth in $B$ by elliptic regularity.
Denote by $q_0$ a global $k$-quasi-harmonic function such that $
\cL^k q_0 = 1$
(such a function exists by Theorem~\ref{t: sol} above).   
Then the 
function $v-c \,   q_0$ is in the Bergman 
space $\bktwo(B)$.  
By 
Theorem~\ref{t: bergman-sobolev}, 
the Bergman space
$\bktwo(B)$ coincides with $\OV{\bktwo(\OV{B})}$.  Hence there 
exists a sequence 
$\{h_j\}$ of $k$-harmonic functions in neighbourhoods of $\OV{B}
$ 
such that 
$$
\lim_{j\to \infty} \bignorm{v-c\,  q_0 - h_j}{{\ld{B}}} 
= 0, 
$$
whence $\{h_j+c\,  q_0\}$  converges to $v$ in $
\ld{B}$, i.e., 
$v$ is in the closure of $\QOVBk$, as required. 

Next we prove \rmii.
Clearly, if $V$ is a function in $\ld{B}$ that may be approximated in
the $\ld{B}$-norm by a sequence of global $k$-quasi-harmonic 
functions, 
then it belongs to the closure of $\Qk{\OV B}$, which, by \rmi, is $
\Qk{B}$.

Conversely, suppose that $v$ is in $\Qk{B}$. Then, by \rmi, it may 
be approximated in
the $L^2$-norm by a sequence $\{u_n\}$ of $k$-quasi-harmonic 
functions in 
$\Qk{\OV B}$. 
Thus, it suffices to show that each of these functions may, in turn, 
be approximated 
in the $L^2(B)$-norm by global $k$-quasi-harmonic functions.  
Set $c_n := \cL^k u_n$.  Denote by $q$ a global $k$-quasi-harmonic function 
such that $\cL^k q = 1$ on $M$.  The function $u_n-cq$ is $k$-
harmonic in a 
neighbourhood of $\OV B$.  Since $\OV B$ has no holes, 
there exists by Lemma~\ref{l: BB} a global $k$-harmonic function 
$w_n$ such that 
$$
\bigmod{u_n-cq-w_n} < 2^{-n}
\qquad \hbox{in $\OV B$}.  
$$
The functions $v_n := w_n+cq$ are the required approximants of 
$v$.  
\end{proof}

\begin{remark}
We note explicitly that if $M$ is a Cartan--Hadamard manifold, 
then 
$$
\OV{\Qk{\OV B}}
= \Qk{B}.
$$
for every geodesic ball $B$. 
\end{remark}

\section{Background on Hardy-type spaces}
\label{s: Background material}

Let $M$ denote a connected $n$-dimensional Riemannian 
manifold
of infinite volume with Riemannian measure $\mu$.
In this section we gather some known facts about the Hardy space 
$\hu{M}$,
introduced by Carbonaro, Mauceri and Meda \cite{CMM1} in the 
setting of 
measured metric spaces of infinite volume (see also \cite{CMM2}
for the case of finite volume), and the Hardy-type spaces $\Xk(M)
$, introduced
in \cite{MMV1} and studied in \cite{MMV2}.  

\begin{definition} \label{def: bounded geometry}
We say that $M$ has $C^\ell$ \emph{bounded geometry} 
if the injectivity radius is positive and the following hold:
\begin{itemize}
\item{} 
if $\ell =0$, then  the Ricci tensor $\Ric$ is bounded from below;
\item{} 
if $\ell$ is positive, then the covariant
derivatives $\nabla^j \Ric$ of the Ricci tensor are uniformly
bounded on $M$ for all $j$ in $\{0,\ldots, \ell\}$.
\end{itemize}
\end{definition}

\begin{Standing assumptions} \label{Ba: on M}
Hereafter we make the following assumptions on $M$:
\begin{enumerate}
\item[\itemno1] $b>0$ ($b$ denotes the bottom of the $L^2$-
spectrum of $\cL$);
\item[\itemno2]
$M$ has $C^\ell$ bounded geometry for some nonnegative 
integer $\ell$. 
\end{enumerate}
\end{Standing assumptions}

\begin{remark} \label{rem: unif ball size cond}
Set $\be =
\limsup_{r\to\infty} \bigl[\log\mu\bigl(B(o,r)\bigr)\bigr]/(2r)$, 
where $o$ is any reference point of $M$
and $B(o,r)$ denotes the geodesic ball with centre $o$ and 
radius~$r$. 
By a result of Brooks $b\leq \be^2$ \cite{Br}.
It is well known that for manifolds with properties \rmi-\rmii\ above
there exist positive constants $\al$ and $C$ such that
\begin{equation} \label{f: volume growth} 
\mu\bigl(B(p,r)\bigr)
\leq C \, r^{\al} \, \e^{2\be \, r}
\quant r \in [1,\infty) \quant p \in M.  
\end{equation}
Furthermore \cite[Remark~2.3]{MMV2} there exists a positive 
constant $C$ such that
\begin{equation} \label{f: lower bound balls}
C^{-1}\,r^n
\leq \mu\bigl(B(p,r)\bigr)
\leq C\,r^n
\quant r\in(0,1] \quant p\in M \,.
\end{equation}
\end{remark}

\medskip
We denote by $\cB$ the family of all balls on $M$.
For each $B$ in $\cB$ we denote by $c_B$ and $r_B$
the centre and the radius of $B$ respectively.  
Furthermore, we denote by $c \, B$ the
ball with centre $c_B$ and radius $c \, r_B$.
For each \emph{scale parameter} $s$ in $\BR^+$, 
we denote by $\cB_s$ the family of all
balls $B$ in $\cB$ such that $r_B \leq s$.

We recall the definitions of the atomic Hardy space $H^1(M)$ and 
its
dual space $BMO(M)$ given in \cite{CMM1}.  We set $s_0:=(1/2)\, 
\Inj(M)$.

\begin{definition}
An $H^1$-\emph{atom} $a$
is a function in $L^2(M)$ supported in a ball $B$
with the following properties:
\begin{enumerate}
\item[\itemno1]
$\intop_B a \wrt \mu  = 0$;
\item[\itemno2]
$\norm{a}{2}  \leq \mu (B)^{-1/2}$.
\end{enumerate}
Given a positive ``scale parameter" $s$, we say that an $H^1$-atom $a$ is at scale $s$ if it is supported in a ball $b$ of $\cB_{s}$. An $H^1$-atom is called \emph{admissible} if it is supported 
in a ball $B$ of $\cB_{s_0}$. 
\end{definition}

\begin{definition} \label{def: H1}
The \emph{Hardy space} $\hu{M}$ is the 
space of all functions~$f$ in $\lu{M}$
that admit a decomposition of the form
\begin{equation} \label{f: decomposition}
f = \sum_{j=1}^\infty c_j \, a_j,
\end{equation}
where $a_j$ are admissible $H^1$-atoms,
and $\sum_{j=1}^\infty \mod{c_j} < \infty$.
The norm $\norm{f}{H^{1}}$
of $f$ is the infimum of $\sum_{j=1}^\infty \mod{c_j}$
over all decompositions \eqref{f: decomposition}  
of $f$.  

We denote with $\hufin{M}$ the vector space of all 
\emph{finite} linear combinations
of admissible $H^1$-atoms, endowed with the norm
$$  
\bignorm{f}{H^1_{\rm{fin}}}=\inf\Big\{\sum_{j=1}^N|c_j|:~f=
\sum_{j=1}^N c_j\,a_j,  
\hbox{\, $a_j$ admissible $H^1$-atom, $N\in\BN$}\Big\}\,.
$$  
\end{definition}  

\noindent
It is known that, under the Standing assumptions~\ref{Ba: on M}, 
the 
$H^1$-norm and the $H^1_{\mathrm{fin}}$-norm are equivalent on 
$\hufin{M}$ 
\cite[Section~4]{MM}.  
\begin{remark}\label{scaleinvH1}
Actually, in the definition of the spaces $H^1(M)$ and $H^1_{\mathrm{fin}}(M)$, the choice of scale is irrelevant. Indeed, in \cite{CMM1} it has been shown that in Definition~\ref{def: H1}  one obtains the same spaces, with equivalent norms, if  admissible atoms are replaced by atoms at any fixed scale $s$. 
\end{remark}
\begin{definition} \label{def: norm BMO}
We define $BMO(M)$ as the space of all locally integrable 
functions~$g$ such that 
$$
\bignorm{g}{BMO} 
:=  \sup_{B\in \cB_{s_0}} \inf_{c \in \BC} \Big(\frac{1}{\mu(B)}
\intop_B \mod{g-c}^2 \wrt\mu\Big)^{1/2} < \infty.  
$$
\end{definition}

\noindent
The Banach dual of $\hu{M}$ 
is isomorphic to $BMO(M)/\BC$ \cite[Thm~5.1]{CMM1}.

\medskip
Now we recall the definition of the \textit{generalised Hardy 
spaces} $\Xh{M}$.
For $\si>0$ denote by $\cU_\si$ the operator 
$\cL\, (\si \cI + \cL)^{-1}$. 
It is known that for every positive integer $k$ the operator $\cU_
\si^k$ 
is injective on $\lu{M}+\ld{M}$ \cite[Proposition~2.4~\rmii]{MMV1}.  

\begin{definition} \label{def: Hardy space}
For each positive integer $k$ and for each $\si>\be^2 -b$
we denote by $\Xh{M}$ the Banach space of all
$\lu{M}$ functions~$f$ such that $\cU_{\si}^{-k}f$ is in $\hu{M}$, 
endowed with the norm
$$
\bignorm{f}{\Xk}
= \bignorm{\cU_{\si}^{-k} f}{H^1}.
$$
\end{definition}

\medskip
\noindent
Clearly $\cU_{\si}^{-k}$ is
an isometric isomorphism between $\Xh{M}$ and $\hu{M}$. 
It is known \cite[Section~3]{MMV1} that the space $\Xh{M}$
does not depend on $\si> \be^2-b$, and that different
values of $\sigma$ give rise to equivalent norms on $\Xh{M}$. 
For later use, it is convenient to assume that $\si>2\be$, and we 
shall denote
$\cU_\si$ simply by $\cU$. 

\begin{definition} \label{def: BMO space}
For each positive integer $k$
we denote by $\Yh{M}$ the Banach dual of $\Xh{M}$.
\end{definition}

\begin{remark}
Since $\cU^{-k}$ is
an isometric isomorphism between $\Xh{M}$ and $\hu{M}$,
the transpose map~$\bigl(\cU^{-k}\bigr)^t$ is an isometric 
isomorphism 
between the dual of $\hu{M}$, i.e., $BMO(M)/\BC$, and $\Yh{M}$.
Hence
$$
\bignorm{\bigl(\cU^{-k}\bigr)^tf}{\Yk}
= \bignorm{f}{BMO/\BC}.
$$
\end{remark}

\noindent
Some properties of $\Xh{M}$ are listed in the introduction (see 
\cite{MMV1,MMV2}).  
The space $\Xh{M}$ admits an atomic decomposition in terms of
``special atoms'' \cite{MMV2}, which we now define.

\begin{definition}
Suppose that $k$ is a positive integer.
An $X^k$-\emph{atom} associated to
the ball $B$ is a function $A$ in $\ld{M}$, supported in $B$, such 
that
\begin{enumerate}
\item[\itemno1]
$\int A\, h\wrt \mu=0 \qquad\forall h\in \QOVBk$;
\item[\itemno2]
$\ds\norm{A}{2}\leq \mu(B)^{-1/2}$.
\end{enumerate}
Note that condition \rmi\ implies that 
$\int_M A\wrt\mu=0$, because $\One_{2B}$ is in $\QOVBk$. 
Given a positive ``scale parameter'' $s$, we say that 
an $X^k$-atom is at scale $s$ if it is 
supported 
in a ball $B$ of $\cB_s$.  As in the case of $H^1$, atoms at scale $s_0$ will 
simply be called \textit{admissible} $X^k$-atoms.
\end{definition}

\noindent
Observe that $X^k$-atoms satisfy an infinite dimensional 
cancellation 
condition.  In \cite{MMV2} we proved the following result.

\begin{theorem}  \label{t: at dec}
Suppose that $k$ is a positive integer and that $M$ has 
$C^{2k-2}$ bounded
geometry (see Definition~\ref{def: bounded geometry}).  Choose a 
``scale parameter'' $s$.  Then the space $\Xh{M}$
is the space of all functions $F$ in $\hu{M}$ that admit a 
decomposition of the form $F= \sum_j c_j\, A_j$, where $\{c_j\}$ is 
a sequence in $\ell^1$ and $\{A_j\}$ is a sequence 
of admissible $X^k$-atoms at scale $s$.  
Furthermore
$$
\norm{F}{X^k}
\asymp \inf\, \bigl\{\sum_{j} \mod{c_j}:  F = \sum_{j} c_j \, A_j,\quad
\rm{where\ }\hbox{$A_j$\ {\rm{are\ }}\   $X^k$-{\rm{atoms at scale 
$s$}} }
\bigr\}.
$$
\end{theorem}

\noindent
Notice that the equivalence of norms above implies that 
``atomic norms'' associated to different ``scale parameters'' $s_1$ 
and $s_2$
are equivalent on $\Xh{M}$.  
As in the definition of $H^1(M)$,  a convenient choice of the scale 
parameter is 
$s_0 := (1/2)\, \Inj(M)$.   This choice of the scale parameter 
will simplify
most of the arguments below, for balls of radius at most~$s_0$ 
have no holes and their boundaries are smooth, 
whence the theory developed in Sections~\ref{s: Quasi} 
and \ref{s: Background material} applies.  
In particular, in view of Proposition~\ref{p: structure CHhperp}
the cancellation condition of an $X^k$-atom $A$  
associated to a ball $B$ may be described in one of the 
following equivalent ways:
\begin{itemize}
\item[(a)]\quad  $\int_B A\, v \wrt\mu=0 
\qquad\forall v\in q_k^2(\OV B)$;
\item[(b)]\quad  $\int_B A\, v \wrt\mu=0 \qquad\forall v\in q_k^2(B)$;
\item[(c)]\quad $\int_B A\, v \wrt\mu=0 \qquad\forall v\in q_k(M)$.
\end{itemize} 

We remark also that $\Xu(M)$ admits an atomic decomposition in 
terms of $X^1$-atoms 
whenever $M$ satisfies mild geometric assumptions, i.e. $M$ has 
positive injectivity
radius, Ricci curvature bounded from below and spectral gap, 
whereas 
if $k\geq 2$, then the atomic decomposition of $\Xh{M}$ requires 
at 
least $C^2$ bounded geometry. 

Next we introduce a norm on the space of finite linear 
combinations
of admissible $X^k$-atoms.  

\begin{definition} \label{def: Xfin}
Suppose that $k$ is a positive integer.   
We denote by $\Xhfin{M}$ the vector space of all finite linear 
combinations
of admissible $X^k$-atoms, endowed with the 
norm 
$$
\bignorm{F}{\Xkfin}
:= \inf \, \Big\{\sum_{j=1}^N \mod{c_j}: F = \sum_{j=1}^N c_j  A_j,
\hbox{$A_j$ admissible $X^k$-atom} \Big\}.
$$
\end{definition}

\begin{remark}\label{scale inv} By combining \cite[Remark 3.5]
{MMV2} with the proof of \cite[Lemma~6.1]{MMV2} one can see 
that any $X^k$-atom $A$ at scale $s>s_0$ can be written as a finite 
linear combination $A=\sum_{j=1}^{N(s)}\lambda_j A_j$ of  admissible $X^k$-atoms 
$A_j$, with $\sum\mod{\lambda_j}\le Cs$. Thus, if in the definition of $\Xhfin{M}$ we replace admissible $X^k$-atom by $X^k$-atoms at any fixed scale $s$, we obtain the same space with an equivalent norm.
\end{remark}

\begin{remark}\label{r: XkUkH1fin}
Notice that $\Xhfin{M}$ is contained in $\cU^k\bigl(\hufin{M}\bigr)$. 
Indeed, for any admissible $X^k$-atom $A$, $\cU^{-k}A$ is 
a multiple of an admissible $H^1$-atom by \cite[Remark 3.5]
{MMV2}. 
Hence, $\cU^{-k}A$ lies in $\hufin{M}$. 
It follows that $A=\cU^{k}\cU^{-k}A$ belongs to 
$\cU^k\bigl(\hufin{M}\bigr)$. 
\end{remark}

\noindent
Clearly
$$
\bignorm{F}{\Xk} \leq \bignorm{F}{\Xkfin}
\quant F \in \Xhfin{M},
$$
so that there is a natural injection of the completion
of $\Xhfin{M}$ into $\Xh{M}$.  We shall show that
this map is an isomorphism of Banach spaces.  

We need a slight variant of the 
``economical decomposition of atoms" proved in \cite[Lemma 6.1]
{MMV2}.

\begin{lemma} \label{l: economical decomposition I}
Suppose that $k$ is a positive integer and that 
$M$ has $C^{2k-2}$ bounded geometry.
If  $a$ is an $H^1$-atom in $\Dom(\cL^k)$, then $\cL^k a$ is in $
\Xh{M}$.
Furthermore, if the support of $a$ is contained  in the ball $B$, 
then there exists a constant $C$ such that
$$  
\bignorm{\cL^k a}{\Xkfin} 
\leq C\, (1+ \,r_B)\, \mu(B)^{1/2}\, \bignorm{\cL^k a}{2}.
$$  
\end{lemma}

\begin{proof} 
It is straightforward to check that the proof of \cite[Lemma~6.1]
{MMV2}
proves the stronger statement above.  
\end{proof}

In the following lemma and elsewhere we shall identify functions in $q_k^2(B)^\perp$ with their extensions to $M$ that vanish outside $B$.
\begin{lemma} \label{l: estimate of Xkfin norm}
Suppose that $k$ is a positive integer and that 
$M$ has $C^{2k-2}$ bounded geometry. 
There exists a constant $C$ such that for every ball $B$
and every $F$ in $\QBkperp$
$$
\bignorm{F}{\Xkfin}
\leq C\, (1+r_B) \, \mu(B)^{1/2} \, \bignorm{F}{2}.
$$
\end{lemma}

\begin{proof}
The function $A:= F/\mu(B)^{1/2} \, \bignorm{F}{2}$ is an $X^k$-
atom
with support contained in $\OV B$.  Hence $a:= \cL^{-k}A/
\bigopnorm{\cL^{-k}}{2}$
is a $H^1$-atom with support contained in $\OV B$.  Note that 
$a$ is in $\Dom(\cL^k)$.   By Lemma~\ref{l: economical 
decomposition I}
$$
\bignorm{\cL^k a}{\Xkfin}
\leq C\, (1+r_B) \, \mu(B)^{1/2} \, \bignorm{\cL^k a}{2},
$$
from which the required estimate follows directly.
\end{proof}

We also need the following result, which provides a 
``nice'' decomposition of $\cU^ka$ for an admissible $H^1$-atom 
$a$ 
in terms of admissible $X^k$-atoms.

\begin{lemma} \label{l: Uka}  
Suppose that $k$ is a positive integer and that
$M$ has $C^{2k-2}$ bounded geometry.
Let $a$ be an admissible $H^{1}$-atom supported in a ball 
$B(p,R)$, 
where $p$ is in $M$ and $R\leq s_0$.  Then there exist a positive 
constant $C$, 
functions $A_i'$ and $A_j''$ such that
$$
\cU^k a= \sum_{i=0}^d A_i'+\sum_{j=1}^{\infty} A_j'',
$$
where $d=[\!\![\log_{4}(3/R)+1]\!\!]$, the series converges in $
\Xh{M}$ and in $L^2(M)$, and
\begin{itemize}
\item[(i)] the function $A_i'$ is supported in $B_i'=B\big(p,(4^i+1)R
\big)$, 
lies in $q^2_k(B_i')^\perp$ and
$$
\bignorm{A_i'}{2}\leq C\,\mu(B_i')^{-1/2}\,4^{-i}\,;
$$
\item[(ii)] the function $A_j''$ is supported in 
$B_j''=B\big(p,j+1\big)$, lies in $q^2_k(B_j'')^\perp$ and 
$$
\bignorm{A_j''}{2}\leq C\,\e^{-2\beta j}.
$$
Moreover, $A_j''$ lies in $\Xhfin{M}$ and there exist 
positive constants $C$ and $\vep$, such that
\begin{equation}\label{ecodec} 
\bignorm{A_j''}{ X^k_{ {\rm{fin}}} }\leq C\,\e^{-\vep j}.
\end{equation} 
\end{itemize}
\end{lemma}  

\begin{proof} 
We write $\cU^k=\cL^k\cR^k$ and 
$\cR^k=r^k(\cD_1)$, where $r(t)=\frac{1}{c^2+t^2}$ and 
$\cD_1=\sqrt{\cL-b+\kappa^2}$. We proceed 
as in the proof of \cite[Lemma 4.2]{MMV2} to construct the 
functions 
$A_i'$ and $A_j''$ with the required properties. 
The functions $A_j''$ are defined as $\cL^k a_j$, where $a_j$ are 
suitable 
multiple of $H^1$-atoms. 
Then the estimate \eqref{ecodec} of the norm of $A_j''$ may be 
obtained by arguing 
as in \cite[Lemma 6.1]{MMV2}. 
\end{proof}

\section{The dual of Hardy-type spaces}
\label{s: dual of H-t s}  
 
In this section we prove our main result, which identifies the dual 
of $\Xh{M}$
to a Banach space of functions on $M$.  
We need more notation and some preliminary results.  
For any open ball $B$ in $M$, 
we denote by 
$\tilde{\pi}_{B,k}: \ld{B} \to \QBkperp$ 
the orthogonal 
projection
onto $\QBkperp$.  We may extend $\tilde{\pi}_{B,k}$ to a 
map $\PBk $ from $\ldloc{M}$
to $\QBkperp$, by setting
$$
\PBk (F) := \tilde{\pi}_{B,k}(F_{|_B})
\quant F \in \ldloc{M}.
$$
where $F_{|_B}$ denotes the restriction of $F$ to $B$.
\begin{proposition} \label{p: global qhfunctions}
Suppose that $s>0$, $G$ is in $\ldloc{M}$
and $\pi_{B,k}(G) = 0$ for every $B$ in $\cB_s$.  Then $G$ is a 
global
$k$-quasi-harmonic function, i.e. it belongs to $q_k(M)$.   
\end{proposition}

\begin{proof}
Observe that $\pi_{B,k} (G) = \tilde{\pi}_{B,k}(G_{|_B})$.
Hence, $G_{|_B}$ is orthogonal to $\QBkperp$, i.e., it belongs 
to $\QBk$.   
Thus, $\cL^k G$ is constant on $B$.  
Now, if $B$ and $B'$ are two balls in $\cB_s$ with nonempty
intersection, $\cL^k G$ is constant both on $B$ and on $B'$,
whence the constant must be the same for the two balls.  
Since $M$ is connected by assumption, $\cL^k G$ is constant on 
$M$, i.e.
it is in $q_k(M)$, as required.
\end{proof}

\begin{definition}  \label{def: the space GBMO}
Suppose that $k$ is a positive integer and $s>0$.   Then 
$GBMO_s^k(M)$
is the vector space of all functions $G$ in $\ldloc{M}$ such that 
$$
\bignorm{G}{GBMO_s^k}
:= \sup_{B\in \cB_s} \mu(B)^{-1/2} \, \bignorm{\pi_{B,k} (G)}{2}
<\infty.
$$
\end{definition}

\noindent
Note that if $r_B < \Inj(M)$, then  by Proposition \ref{p: structure CHhperp} \rmii
$$
\bignorm{\pi_{B,k} (G)}{2} 
= \inf_{V\in q_k(M)} \Big[\int_B \bigmod{G-V}^2 \wrt \mu 
\Big]^{1/2}.
$$
Loosely speaking, if $s< \Inj(M)$, then the space $GBMO_s^k(M)$ 
consists of all 
locally square-integrable functions $G$,
which are ``well approximated'' on each ball $B$ in $\cB_s$
by global $k$-quasi-harmonic functions.
If we interpret constants as
$0$-quasi-harmonic functions, we may say that 
$BMO(M)$ functions are  those locally square-integrable 
functions,
which are ``well approximated'' on each ball $B$ in $\cB_s$
by $0$-quasi-harmonic functions.  Thus, functions in 
$GBMO_s^k(M)$ 
may be considered as generalisations of functions in $BMO(M)$, 
a fact which 
partially justifies the notation.  
\par

Henceforth, we shall consider the spaces $GBMO_s^k(M)$ only for $s<\Inj(M)$ and we shall write $GBMO^k(M)$ instead of $GBMO_{s_0}^k(M)$, where $s_0=\frac{1}{2}\Inj(M)$.  We shall prove later that \emph{  if $s$ is less than $\Inj(M)$} then the spaces $GBMO_s^k(M)$ do not depend on $s$ and that all the norms $\norm{\cdot}{GBM_s^k}$, $0<s<\Inj(M)$, are equivalent (see Corollary \ref{}).

\par
Obviously, $\norm{\cdot}{GBMO_s^k}$ vanishes on $q_k(M)$ and, 
by Proposition~\ref{p: global qhfunctions}, 
it defines a norm on the quotient space $GBMO_s^k(M)/q_k(M)$.  
Note that if $k\le \ell$ then a function $G$ in $GBMO_s^k(M)$ is 
also in $GBMO_s^\ell(M)$, 
for 
\begin{equation}\label{l,k}
\bignorm{\pi_{B,\ell} (G)}{2} 
\leq \bignorm{\pi_{B,k} (G)}{2}.
\end{equation}
In particular, any representative of a class in $BMO(M)/\BC$,
represents also a class in $GBMO_s^\ell(M)/q_\ell(M)$.  
\par

The main result of this section (Theorem \ref{t: main} below) is that 
the dual $\Yh{M}$ of $\Xh{M}$ can be identified with 
$GBMO^k(M)/q_k(M)$ via the map $\iota$ that to each 
coset $G+q_k(M)$ in $GBMO^k(M)/q_k(M)$ 
associates the functional $\iota(G+ q_k(M))$ on $\Xhfin{M}$ 
defined  by
\begin{equation} \label{f: def of iota on atoms}
\iota(G+ q_k(M)) (F)
:= \int_M F \, G \wrt \mu
\quant F \in \Xhfin{M}.
\end{equation}
It is straightforward to check that the integral above does not 
change
if we replace~$G$ by any other representative of the coset $G
+q_k(M)$.
At this point it is by no means clear that the functional 
$\iota(G+ q_k(M))$ extends to a continuous linear functional on $
\Xh{M}$.
We shall prove that this is indeed the case and 
that $\iota$ extends to a Banach space isomorphism between 
$GBMO^k(M)/q_k(M)$
and $\Yh{M}$ (see Theorem~\ref{t: main}).

\noindent
To prove this result it is useful to introduce another space 
that will play also a role in the characterization of the dual of $X^k_{\mathrm{fin}}(M)$ in the next section.

\begin{definition} \label{def: the space Y}
Suppose that $k$ is a positive integer. 
We define $\bYh{M}$ to be the space of all families of functions
$\bG := \{G_B: B\in\cB\}$ such that 
\begin{enumerate}
\item[\itemno1]
$G_B$ is in $\QBkperp$ and 
$\Pik{B}(G_{B'}) = G_B$ for all $B,B'\in\cB$ such that $B\subset 
B'$;
\item[]
\item[\itemno2]
$\displaystyle{
\norm{\bG}{\bYk}
:=  \sup_{B\in \cB_{s_0}} \Big(\frac{1}{\mu(B)}\int_B|G_B|^2 \wrt
\mu\Big)^{1/2}
< \infty.
}$
\end{enumerate}
\end{definition}

\noindent
It is straightforward to check that if $G$ is in $GBMO^k(M)$, then 
the
collection $\bG := \{\pi_{B,k}(G): B \in \cB\}$ is in $\bYk(M)$, and 
$$
\norm{\bG}{\bYk} = \norm{G}{GBMO^k}.
$$
Conversely, given $\bG = \{G_B: B \in \cB\}$ 
in $\bYk(M)$, it is not clear \textit{a priori} 
whether there exists $G$ in $GBMO^k(M)$ such that $G_B = 
\pi_{B,k}(G)$ for 
every $B$ in $\cB$.  In Corollary~\ref{c: G and bG} we shall prove 
that this is 
indeed the case, following a somewhat long detour.  
It would be nice to have a more direct proof of this fact.

\begin{definition}
Given a function $h:\cB \rightarrow \BC$ and a complex number $
\al$,  
we say that $\lim_{ B } h(B)=\al$ if for every $\vep>0$ 
there exists a ball $B_{\vep}$ such that 
$$
|h(B)-\alpha|<\vep \qquad \forall B\in\cB ~{\rm{such~that~}} 
B_{\vep}\subset B\,.
$$  
\end{definition}

\noindent
Fix a reference point $o$ in $M$, and, 
for every positive integer $m$, denote by $B_m$ the ball with 
centre $o$
and radius $m$.

\begin{lemma} \label{l: dual of Xhfin}
Suppose that $k$ is a positive integer and that $M$ has 
$C^{2k-2}$ bounded geometry. 
The following hold:
\begin{enumerate}
\item[\itemno1]
for every ${\bG}$ in $\bYk(M)$ the linear functional $\la_{{\bG}}$ 
on $\Xhfin{M}$, 
defined by 
\begin{equation} \label{f: lambdasuXfin}
\la_{\bG}(F)
= \lim_{B} \int_M F \, G_B \wrt \mu
\quant F \in \Xhfin{M},
\end{equation}
is continuous on $\Xhfin{M}$ and 
$ \norm{\la_{\bG}}{(X^k_{\rm{fin}})^*} 
\leq \norm{{\bG}}{\bYk}$;
\item[\itemno2]
there exists
a positive constant $C$ such that for every $B$ in $\cB$ and 
for every $\bG$ in $\bYk(M)$ 
$$
\Big(\frac{1}{\mu(B)}\int_B|G_B|^2\wrt\mu\Big)^{1/2} 
\leq C  \bignorm{\bG}{\bYk}  (1+r_B);
$$
\item[\itemno3]
for every admissible $H^1$-atom $a$ and for every $\bG$ in $
\bYh{M}$, the limit
$$
\la(\cU^k a)
:= \lim_{m\rightarrow \infty } (\cU^k a, G_{B_m})
$$ 
exists.  Furthermore, there exists a positive constant $C$ such 
that 
$$
\mod{\la(\cU^k a)}
\leq C\, \norm{\bG}{\bYk};  
$$ 
\item[\itemno4]
for every $\bG$ in $\bYk(M)$ the linear functional $\la_{\bG}$
on $\Xhfin{M}$, defined in~\rmi,
extends to a continuous linear functional on $\Xh{M}$, and there 
exists
a constant $C$, independent of $\bG$, such that
$\norm{\la_{\bG}}{(X^k)^*} \leq C\,\norm{\bG}{\bYk}$;
\item[\itemno5]
for every $G$ in $GBMO^k(M)$ the linear functional $\iota(G+ 
q_k(M))$
on $\Xhfin{M}$, defined by
$$
\iota \big(G+ q_k(M) \big)(F)
= \int_M F \, G \wrt \mu
\quant F \in \Xhfin{M},
$$
extends to a unique continuous linear functional on $\Xh{M}$, and 
the
map $\iota: GBMO^k(M)/q_k(M) \to \Yk(M)$ that associates to the
coset $G+ q_k(M)$ the extension of $\iota \big(G+ q_k(M) \big)$
described above is a continuous linear map. 
\end{enumerate}
\end{lemma}

\begin{proof}
First we prove \rmi.  Note that the limit 
$$
\lim_{B} \int_M F \,\, G_B \wrt \mu
\quant F \in \Xhfin{M}
$$
exists, for the support of $F$ is contained in a ball.
Thus, $\la_{\bG}(F)$ is well defined.  
Suppose that $A$ is an $X^k$-atom with support contained in 
a ball $B$ in $\cB_{s_0}$.  We have that
$$
\begin{aligned}
\mod{\la_{\bG}(A)}
& =    \Bigmod{\int_B A\, G_B \wrt \mu} \\
& \leq \bignorm{A}{2} \, \bignorm{ G_B}{2}  \\
& \leq  \mu(B)^{-1/2}\,  \bignorm{ G_B}{2} \\
& \leq \bignorm{{\bG}}{\bYk}.
\end{aligned}
$$
Therefore, if $F=\sum_{j=1}^N c_j A_j$ is in $\Xhfin{M}$, then
$$
\mod{\la_{\bG}(F)}
\leq   \bignorm{{\bG}}{\bYk} \,\sum_{j=1}^N \mod{c_j}.
$$
We now take the infimum of both sides with respect to all finite 
representations of~$F$, and obtain 
$$
\mod{\la_{\bG}(F)}
\leq  \bignorm{{\bG}}{\bYk} \bignorm{F}{\Xkfin} 
\quant F \in \Xhfin{M},
$$
as required to conclude the proof of \rmi.

Next we prove \rmii.
If $r_B \leq s_0$, then the required estimate follows directly
from the definition of the space $\bYk(M)$.

Suppose that $r_B >s_0$.  Denote by $\la_{\bG}$ the continuous 
linear
functional on $\Xhfin{M}$ associated to ${\bG}$ as in \rmi.  Then
$$
\begin{aligned}   
\mod{(F,  G_B)}
& = \mod{\la_{\bG}(F)} \\
& \leq \bignorm{\la_{\bG}}{(\Xkfin)^*} \, \bignorm{F}{\Xkfin} \\
& \leq C\, \bignorm{\la_{\bG}}{(\Xkfin)^*} \, 
     (1+r_B) \, \mu(B)^{1/2}\, \bignorm{F}{2} 
\quant F \in \QBkperp,
\end{aligned}
$$
by Lemma \ref{l: estimate of Xkfin norm}. 
By taking the supremum of both sides over all $F$ in $\QBkperp$
such that $\norm{F}{2} =1$, we obtain that 
$$
\bignorm{G_B}{2}
\leq C\, \bignorm{\la_{\bG}}{(\Xkfin)^*} \, (1+r_B) \, \mu(B)^{1/2},
$$
which, in view of \rmi, implies the required conclusion.

Now, we prove \rmiii\ for $k=1$.   The proof in the case where $k
\geq 2$
is similar and is omitted. 
Suppose that the atom $a$ is supported in the ball $B(p,R)$ with 
$R\leq s_0$.
Observe preliminarily that both $\cU a$ and $G_{B_m}$ (recall 
that 
$B_m$ denotes the ball with centre $o$ and radius $m$) are in $
\ld{M}$, so that
the inner product $(\cU a,G_{B_m})$ in the statement makes 
sense. 
By Lemma \ref{l: Uka}, we may write  
$$
\cU a 
= \sum_{i=0}^{d} A_i' + \sum_{j=1}^\infty \, A_j'', 
$$ 
where the series $\sum_{j=1}^\infty A_j''$ converges in $\ld{M}$, $A_i'$ is supported in $B'_i$ and $A_i''$ is supported in $B''_i$.
Therefore  
$$
(\cU a,G_{B_m})
= \sum_{i=1}^d (A_i', G_{B_m}) + \sum_{j=1}^\infty  (A_j'', 
G_{B_m}).
$$

We \textit{claim} that
$$  
\lim_{m\to \infty} (\cU a,G_{B_m})
= \sum_{i=1}^d \, \lim_{m\to \infty} (A_i', G_{B_m}) 
   + \sum_{j=1}^\infty  \, \lim_{m\to \infty} (A_j'', G_{B_m}).
$$  
To prove this we set
$
c_{j,m}
:=(A_j'', G_{B_m})
$
and we show that $\sup_{m} \mod{c_{j,m}}$ is a summable 
sequence, whence the result will follow by the Dominated 
Convergence Theorem.      We
denote by $\de$ the distance between $p$ and $o$ and consider 
the three cases
$$
m< \de-j-1, \quad
\de-j-1 \leq m \leq  \de+j+1 \quad
\hbox{and}\quad  
m> \de+j+1
$$
separately.  

In the first case, $B_j''\cap B_m = \emptyset$, so that $c_{j,m} = 
0$.  

In the second case, $B_j''\cap B_m \neq \emptyset$, and
\begin{equation} \label{f: second case}
\begin{aligned}
\mod{c_{j,m}}
& \leq \norm{A_j''}{2} \, \norm{G_{B_m}}{2}   \\
& \leq C\, \e^{-2\be j} \, m \,\mu\big(B_m\big)^{1/2} \, 
      \norm{\bG}{\bYone}   \\
& \leq C\, \e^{-2\be j} \, (\de+j+1)\, \mu(B_{\de+j+1})^{1/2} \, 
      \norm{\bG}{\bYone}   \\
& \leq C\, \e^{-\vep j} \,  \norm{\bG}{\bYone},
\end{aligned}
\end{equation}
for some positive $\vep$.  Here we have used the estimate 
of the $L^2$-norm of $A_j''$ and $G_B$ given in Lemma \ref{l: 
Uka}~(ii) 
and \rmii\ above, respectively, and inequality (\ref{f: volume 
growth}). The constant
$C$  is independent of $j$ and $m$, 
but may depend on the point $p$.  

Finally, in the third case $B_m \supset B_j''$. 
Since $A_j''$ is in $Q_1(B_j'')^\perp$, 
\begin{equation} \label{f: third case}
\begin{aligned}
\mod{c_{j,m}}
& = \mod{(A_j'', \pi_{B_j'',1} (G_{B_m}))}  \\
& \leq \norm{A_j''}{2} \, \norm{\pi_{B_j'',1}(G_{B_m})}{2}   \\
& \leq C\, \e^{-2\be j}  \, j\, \mu(B_j'')^{1/2} \,  \norm{\bG}{\bYone}   
\\
& \leq C\, \e^{-\vep j} \,  \norm{\bG}{\bYone},
\end{aligned}
\end{equation}
for some positive $\vep$.  Here we have applied again the 
estimate 
of the $L^2$-norms of $A_j''$ and of $G_{B_j''}$ given in Lemma 
\ref{l: Uka}~(ii)
and in \rmii\ above, respectively, and 
$C$ is a constant which is independent of $j$ and $m$, 
but may depend on the point $p$.  

This completes the proof that $\sup_{m} \mod{c_{j,m}}$ is a 
summable sequence. 
To conclude the proof of the claim, it remains to observe that 
$$
\lim_{m\to \infty} (A_i', G_{B_m})
\qquad\hbox{and}\qquad
\lim_{m\to \infty} (A_i'', G_{B_m})
$$
exist, because the sequences
$m\mapsto (A_i', G_{B_m})$ and $m\mapsto (A_i'',G_{B_m})$ are 
eventually constant.  
To conclude the proof of point \rmiii \ of the lemma, it remains to 
prove the estimate 
in the statement.   
Since $A_i'$ and $A_j''$ are in $q^2_1(B_i')^\perp$ and 
$q^2_1(B_j'')^\perp$, respectively, we get
$$
\begin{aligned}
\mod{\la(\cU a)} 
& \leq \sum_{i=1}^d \bigmod{\lim_{m\rightarrow \infty}
(A_i',G_{B_m})} 
         + \sum_{j=1}^\infty  \bigmod{\lim_{m\rightarrow \infty}
(A_j'',G_{B_m})} \\
& =   \sum_{i=1}^d \bigmod{(A_i',G_{B_i'})}  
         +  \sum_{j=1}^\infty  \bigmod{(A_j'',G_{B_j''})}\\
& \leq \sum_{i=1}^d \norm{A_i'}{2}\, \norm{G_{B_i'}}{2} 
        + C\, \norm{\bG}{\bYone}
           \sum_{j=1}^\infty\e^{-2\be j}  \, j\, \mu(B_j'')^{1/2} \\
& \leq \sum_{i=1}^d  4^{-i} \, \mu(B_i')^{-1/2}\,  \norm{G_{B_i'}}{2}
          + C\,\norm{\bG}{\bYone} \\
& \leq C\,\norm{\bG}{\bYone},
\end{aligned}
$$
where we have applied Lemma \ref{l: Uka} and
$C$ is indipendent of $a$, as required. 

To prove \rmiv\ observe that, given $\bG$ is in $\bYh{M}$, the 
linear functional 
$a\mapsto \lim_{m\rightarrow \infty} (\cU^k a, G_{B_m})$
extends, by \rmiii, to a unique linear functional on $\hufin{M}$ that
is uniformly bounded on atoms.  Thus, it extends
to a unique continuous linear functional, $\ell$ say, on $\hu{M}$ 
(see \cite[Theorem 4.1]{MM}).
In particular, 
$$  
\norm{\ell}{(H^1)^*}
\leq C \sup \, \{ \mod{\ell(a)}: \hbox{$a$ $H^1$-atom} \}
\leq C\, \norm{\bG}{\bYk},
$$
where $C$ is the same as in \rmiii.   
Since $\cU^{-k}$ is an isometry between $\Xh{M}$ and $\hu{M}$,
the linear functional $\ell\circ \cU^{-k}$ is in $\Yh{M}$, and 
$$
\norm{\ell\circ \cU^{-k}}{\Yk} 
= \norm{\ell}{(H^1)^*} 
\leq C \, \norm{\bG}{\bYk}.
$$
Furthermore, if $F$ is in $\cU^k \hufin{M}$, then $F= \cU^kf$
for some $f$ in $\hufin{M}$, and 
$$
\begin{aligned}
(\ell\circ \cU^{-k}) (F)
& = (\ell\circ \cU^{-k}) (\cU^k f) \\
& = \ell(f) \\
& = \lim_{m\rightarrow \infty} ( \cU^kf \,, G_{B_m})\,,
\end{aligned}
$$
and 
\begin{equation}\label{eq: boundUkH1fin}
\begin{aligned}
\mod{(\ell\circ \cU^{-k}) (F)}
& = \mod{\ell(f)} \\
& \leq C \, \norm{f}{H^1} \, \norm{\bG}{\bYk} \\
& =    C \, \norm{F}{\Xk} \, \norm{\bG}{\bYk} 
\quant F \in \cU^k \bigl(\hufin{M}\bigr).
\end{aligned}
\end{equation}
Since $\Xhfin{M}\subset \cU^k\bigl(\hufin{M}\bigr)$ 
(see Remark \ref{r: XkUkH1fin}), 
the space $\cU^k \bigl(\hufin{M}\bigr)$ is dense in $\Xh{M}$. 
Then, by \eqref{eq: boundUkH1fin}, $\ell\circ \cU^{-k}$ 
extends to a unique bounded functional $\la_{\bG}$ on $\Xh{M}$,  
defined by \eqref{f: lambdasuXfin} on $\Xhfin{M}$,
such that $\norm{\la_{\bG}}{(X^k)^*} \leq C\,\norm{\bG}{\bYk}$. 

Finally, we prove \rmv.  Pick $G$ in $GBMO^k(M)$ and set $\bG:= 
\{ \pi_{B,k}(G): B \in \cB \}$.  Clearly,
$\iota\big(G+q_k(M)\big)$ agrees on $\Xhfin{M}$
with the functional $\la_{\bG}$, defined in \eqref{f: lambdasuXfin}.
By \rmiv, $\la_{\bG}$ extends uniquely 
to a continuous linear functional on $\Xh{M}$   
(for $\Xhfin{M}$ is norm dense in $\Xh{M}$).  The required norm
 estimate
follows then from \rmiv.
\end{proof}
\begin{corollary}\label{indip da s}
All the spaces $GBMO_s^k(M)$, $0<s<\Inj(M)$, coincide 
and all 
the norms $\norm{\cdot}{GBMO_s^k}$ are equivalent.
\end{corollary}
\begin{proof}
Suppose that  $s_1\le s_2$. It is obvious from  Definition \ref{def: the space GBMO} that 
$GBMO_{s_2}^k(M)\subseteq GBMO_{s_1}^k(M)$
and $\norm{G}{GBMO_{s_1}^k}\le \norm{G}{GBMO_{s_2}^k}$.\par 
Assume next that $G\in GBMO_{s_1}^k(M)$ and  
for every ball $B \in\cB$ define 
 $G_B=\pi_{k,B}(G)$. Let $\lambda_G$ be the linear functional 
 on $X_{\mathrm {fin}}^k(M)$ defined by
$$
\lambda_G(F)=\lim_B \int F\, G_B\wrt\mu.
$$
Now,  arguing as in the proof of Lemma \ref{l: dual of Xhfin} \rmi, but using  $X^k$-atoms with support contained in balls of  radius less than $s_1$ instead of $s_0$ (see Remark \ref{scale inv}), we  obtain that
$$
\mod{\lambda_G(F)} \le\, C\, \norm{G}{GBMO_{s_1}^k}\ \norm{F}{X_{\mathrm {fin}}^k(M)} \qquad\forall F\in X_{\mathrm {fin}}^k(M).
$$
Hence, arguing as in the proof of part \rmii\  of Lemma \ref{l: dual of Xhfin}, we obtain that for all balls $B\in\cB$
$$
\begin{aligned}
\norm{G_B}{2}&\le\, C  \norm{\lambda_G}{(X_{\mathrm{fin}}^k)^*}\,C\, 
(1+r_B)\, \mu(B)^{1/2}\\
&\le\, C\norm{G}{GBMO_{s_1}^k}\ (1+r_B)\, \mu(B)^{1/2}.
\end{aligned}
$$
Thus
$$
\norm{G}{GBMO_{s_2}^k}=\sup_{B\in\cB_{s_2}} \mu(B)^{-1/2}\ \norm{G_B}{2}\le\,C \,(1+s_2)\norm{G}{GBMO_{s_1}^k}.
$$
This shows that $GBMO_{s_1}^k(M)\subseteq GBMO_{s_2}^k(M)$ and that the two norms are equivalent.
\end{proof}
To prove that the map $\iota: GBMO^k(M)/q_k(M) \to \Yk(M)$ is an isomorphism, we need  a regularity result for solutions in $L^2_{loc}(M)$ of the equation $\cL^k u=g$ with $g$ in $BMO(M)$.

\begin{proposition} \label{p: cont iota}
For every $g\in BMO(M)$ and for every positive integer $k$ any solution  $U_{k,g}$  of the 
equation $\cL^k u=g$ is in $GBMO^k(M)$ and
$$
\norm{U_{k,g}}{GBMO^k}\le \opnorm{\cL^{-k}}{2}\ \norm{g}{BMO}.
$$.
\end{proposition}

\begin{proof}
Suppose that $B$ is in $\cB_{s_0}$, and denote by 
$\vp$ a smooth function with compact support that is equal to $1$ 
in a neighbourhood of $\OV B$.  
Then the function $\cL^{-k} (\vp\, g)$
satisfies the equation $\cL^k u = \vp \, g$ on $M$.  Hence
$U_{k,g} - \cL^{-k} (\vp\, g)$ satisfies the equation $\cL^k u = (1-
\vp) \, g$,
so that $U_{k,g} - \cL^{-k} (\vp\, g)$ is $k$-harmonic in a 
neighbourhood of $\OV B$. 
Therefore 
$$
\pi_{B,k}\bigl(U_{k,g} - \cL^{-k} (\vp\, g) \bigr) = 0,
$$
whence
$$
\begin{aligned}
\norm{\pi_{B,k}(U_{k,g})}{2} 
& =  \norm{\pi_{B,k}(\cL^{-k} (\vp\, g))}{2} \\
& =  \sup \bigl\{ \bigmod{\bigl(F,\cL^{-k} (\vp\, g)\bigr)}: 
       F \in q_k^2(B)^\perp, \, \norm{F}{2} =1 \bigr\}.
\end{aligned}
$$
Since $\cL^{-k}$ is self adjoint and $\cL^{-k}F$ is 
a multiple of an $H^1$-atom supported in $\OV B$ (see 
\cite[Remark~3.5]{MMV2}), 
\begin{equation}
\begin{aligned}
\bigl(F,\cL^{-k} (\vp\, g)\bigr)
& = (\cL^{-k} F,\vp\, g)  \\
&=\int_B \cL^{-k} F\,(g-g_B)\wrt\mu \label{indfi}\\
& = \Big(\cL^{-k} F, \vp (g-g_B)\Big) \\
&=  \Big( F, \cL^{-k}(\vp (g-g_B))\Big).
\end{aligned}
\end{equation}
Hence 
$$
\mod{\Big(F, \cL^{-k}(\vp g)\Big)}\le \bigopnorm{\cL^{-k}}{2} \, 
\norm{\vp(g-g_B)}{2}.
$$
Since by the second equality in (\ref{indfi}), $\Big( F, \cL^{-k}(\vp g)
\Big)$ does not depend on $\vp$, we obtain that
$$
\bigmod{\bigl(F,\cL^{-k} (\vp\, g)\bigr)}
 \leq \bigopnorm{\cL^{-k}}{2} \, \norm{g-g_B}{L^2(B)}.
$$
whence
$$
\sup \bigl\{ \bigmod{\bigl(F,\cL^{-k} (\vp\, g)\bigr)}:
       F \in q_k^2(B)^\perp, \, \norm{F}{2} =1 \bigr\}  
\leq \bigopnorm{\cL^{-k}}{2} \, \norm{g-g_B}{L^2(B)}.
$$
By combining this inequality with the formulae above, we 
conclude that
$$
\begin{aligned}
\bignorm{U_{k,g}}{GBMO^k}
& = \sup_{B\in \cB_{s_0}} \mu(B)^{-1/2} \, \norm{\pi_{B,k}(U_{k,g})}
{2}  \\
& \leq \opnorm{\cL^{-k}}{2} \, \sup_{B\in \cB_{s_0}} \mu(B)^{-1/2} \, 
    \norm{g-g_B}{L^2(B)} \\
& = \opnorm{\cL^{-k}}{2} \, \norm{g}{BMO}, 
\end{aligned}
$$
as required.
\end{proof}

\begin{theorem} \label{t: main}
Suppose that $k$ is a positive integer and that $M$ has 
$C^{2k-2}$
bounded geometry.  Then the map $\iota$ 
(see Lemma~\ref{l: dual of Xhfin}~\rmv)
is a Banach space isomorphism between $GBMO^k(M)/q_k(M)$ 
and  $\Yk(M)$.  
\end{theorem}

\begin{proof}
By Lemma~\ref{l: dual of Xhfin}~\rmv\ the map $\iota$ is 
continuous. We shall prove that 
$\iota$ is bijective.  The required conclusion will then 
follow from (a standard consequence of) the Open Mapping 
Theorem.

First we show that $\iota$ is \emph{injective}.  Suppose that $G$ 
is a function
in $GBMO^k(M)$ such that $\iota (G+q_k(M)) = 0$.  In particular, 
$$
\iota (G+q_k(M))(A) 
=\int_M A \, \, G \wrt \mu 
= 0 
$$
for every $X^k$-atom $A$.   This implies that $\pi_{B,k}(G) = 0$ 
for every 
ball $B$ with radius $\leq s_0$.  
By Proposition~\ref{p: global qhfunctions}, $G$ is in $q_k(M)$, as 
required. 
\par
Next we prove that $\iota$ is surjective. Suppose that $\la$ is in 
$\Yh{M}$.
Since $\big(\cU^{-k}\big)^t$ is an isomorphism between $BMO(M)/
\BC$ and
$\Yh{M}$, there exists a unique coset $g+\BC$ in $BMO(M)/\BC$ 
such that 
$$
\big(\cU^{-k}\big)^t(g+\BC) = \la.
$$
Therefore, for every $X^k$-atom $A$
$$
\begin{aligned}
\la\big(A\big)
& = \prodo{A}{\big(\cU^{-k}\big)^t(g+\BC)}  \\
& = \prodo{\cU^{-k}(A)}{g+\BC},
\end{aligned}
$$
by definition of transpose operator.  The pairing in the first line 
is the duality between $\Xh{M}$ and $\Yh{M}$ and that in the 
second
is the duality between $\hu{M}$ and $BMO(M)/\BC$. 
Since $A$ is in $\ld{M}$, 
$$
\cU^{-k}(A)
= \sum_{j=0}^k  \,\binom{k}{j}\, \si^{j}\cL^{-j}A,
$$
whence
$$
\prodo{\cU^{-k}(A)}{g+\BC}
= \int_M \Big( \sum_{j=0}^k  \,\binom{k}{j}\, \si^{j} \cL^{-j}A \Big)\, \, 
g \wrt \mu,
$$
for $\cL^{-j} A$, $j=0,1,\ldots, k$ are in $\hufin{M}$.  
Now, denote by $\vp$
a smooth function with compact support which is equal to $1$ in a
neighbourhood of the support of $A$, and let $U_{g,j}$ denote any global 
solution of the 
equation $\cL^j u = g$.   We remark that $U_{g,j}\in GBMO^j(M)\subseteq GBMO^k(M)$ by Proposition \ref{} and the remark preceeding (\ref{l,k}). Then 
$$
\begin{aligned}
\int_M  \cL^{-j} A \, \, \, g \wrt \mu
&  = \int_M  \cL^{-j} A \, \, \,\cL^{j} U_{g,j} \wrt \mu  \\
&  = \int_M  \cL^{-j} A \, \,\,\cL^{j}(\vp \,U_{g,j}) \wrt \mu  \\
&  = \int_M   A \, \, \vp \,U_{g,j} \wrt \mu  \\
&  = \int_M   A \, \,U_{g,j} \wrt \mu.  
\end{aligned}
$$
We have used the fact that the support of $\cL^{-j}A$ is contained in the support of $A$ in the second equality, and the self adjointness of $\cL$ in the third equality.  
Now, define $U=\sum_{j=0}^k \binom{k}{j}\,\si^j\ U_{g,j}$. The function $U$ is in $GBMO^k(M)$ by the remark preceeding (\ref{l,k}). By combining the formulae above, we see that 
$$
\begin{aligned}
\la(A) 
& = \int_M A \, U\wrt\mu\\
&=\iota(U+q_k(M))(A).
\end{aligned}
$$
This completes the proof of the surjectivity of $\iota$, and of the 
theorem.   
\end{proof}
\begin{remark}\label{} We observe that in the proof of Theorem \ref{} we have actually shown the commutativity of the following diagram
\par
\unitlength1mm
\begin{picture}(100,48)(-55,-20)  \label{pic: diagramma}
\put(-33,15){$BMO(M)/\BC$}
\put(-13,-16){$GBMO^k(M)/q_k(M)$}
\put(-10,16){\vector(1,0){35}}
\put(28,15){$Y^k(M)$}
\put(-19,12){\vector(1,-1){22}}
\put(10,-9){\vector(1,1){22}}
\put(-14,0){$\cJ$}
\put(23,-1){$\iota$}
\put(1,19){$\big(\cU^{-k}\big)^t$}
\end{picture}
\par \noindent
where $\cJ$ is the map $g+\BC\mapsto \sum_{j=0}^k \binom{k}{j} \ \sigma^j U_{g,j}+q_k(M)$.
\end{remark}
Now we draw a few consequences of Theorem \ref{t: main}. 

\begin{corollary} \label{c: Ik+siTk}
If $g$ is a 
function 
in $BMO(M)$ such that $\cL g+\si g=const$  then $g$ is constant.
\end{corollary}
\begin{proof}
Let $U_{g,1}$ be a global solution of the equation $\cL u=g$. Then $g+\sigma U_{g,1}\in q_1(M)$. Thus 
$$
\int_M (g+\sigma U_{g,1})\ A\wrt \mu=0 
$$
for all $X^1$-atoms $A$. In the proof of Theorem~\ref{t: main} we have shown that 
$$
\langle A,(\cU^{-1})^t(g+\BC)\rangle= \langle \cU^{-1} A, g+\BC\rangle=\int_M(g+\sigma U_{g,1})\ A\wrt \mu=0.
$$
Thus $(\cU^{-1})^t(g+\BC)=0$ and the conclusion follows, since $(\cU^{-1})^t$ is an isomorphism.
\end{proof}

\begin{corollary} \label{c: G and bG}
For every $\bG = \{G_B: B \in \cB\}$ 
in $\bYk(M)$, there exists $G$ in $GBMO^k(M)$ such that $G_B 
= \pi_{B,k}(G)$ for 
every $B$ in $\cB$.  Furthermore, $\norm{\bG}{\bYk} \asymp 
\norm{G}{GBMO^k}$. 
\end{corollary}

\begin{proof}
Suppose that $\bG$ is in $\bYk(M)$.  By Lemma~\ref{l: dual of 
Xhfin}~\rmiv,
the linear functional $\la_{\bG}$, defined by
$$
\la_{\bG}(F)
= \lim_{B'} \int_M F \,\, G_{B'} \wrt \mu
\quant F \in \Xhfin{M},
$$
is in $\Yh{M}$.
Theorem~\ref{t: main} then ensures the existence of a function $G
$ in 
$GBMO^k(M)$ such that 
$
\la_{\bG} = \iota\big(G+q_k(M)\big),
$  
and 
$$
\iota\big(G+q_k(M)\big) (F)
= \int_M F\,\, G \, \wrt \mu
\quant F \in \Xhfin{M}.
$$
Therefore, given a ball $B$, for every ball $B'$ containing $B$
and for every (possibly not admissible) $X^k$-atom $A$ 
associated to $B$ 
we have that 
$$
\int_M A \,\, G_{B'} \wrt \mu
= \int_M A\,\, G \, \wrt \mu.  
$$
It follows that $\pi_{k,B}(G_{B'}) = \pi_{k,B}(G)$.  But
$\pi_{k,B}(G_{B'}) = G_B$, because $\bG$ is in $\bYk(M)$,
and the required formula follows.  
The equivalence of the norms of $\bG$ and $G$ is an obvious 
consequence of the definition of the ``norms'' of $\bYk(M)$ and 
$GBMO^k(M)$.
\end{proof}

\section{The dual of $\Xhfin{M}$}
\label{s: dual of Xkfin}  

A noteworthy consequence of the theory developed in 
Section~\ref{s: dual of H-t s} is the fact, proved in the next 
theorem, 
that $\Xhfin{M}$ and $\Xh{M}$ have isomorphic duals. 

\begin{theorem} \label{t: dual of Xk}
Suppose that $k$ is a positive integer and $M$ has 
$C^{2k-2}$ bounded geometry. 
The dual of $\Xhfin{M}$ is isomorphic to $GBMO^k(M)/q_k(M)$.
The continuous linear functionals on $\Xhfin{M}$ are
precisely those of the form
$$
\la_G (F)
= \int_M F \, G \wrt \mu
\quant F \in \Xhfin{M}
$$
for $G$ in $GBMO^k(M)$.  Furthermore  
$\norm{\la}{(X^k_{\mathrm{fin}})^*} \asymp \norm{G}{GBMO^k}$.
\end{theorem}

\begin{proof}
By Lemma~\ref{l: dual of Xhfin}~\rmi\ the linear functional 
$\la_{G}$ is continuous on $\Xhfin{M}$ and 
$\norm{\la_G}{(X^k_{\rm{fin}})^*} \leq \norm{G}{GBMO^k}$.

Conversely, given a continuous linear functional~$\la$ on $
\Xhfin{M}$, 
 for every $B$ in $\cB$ the restriction of~$\la$ to $\QBkperp$ is 
in $\bigl(\QBkperp\bigr)^*$.   Indeed, 
\begin{equation} \label{f: estimate la}
\begin{aligned}
\mod{\la(F)}
& \leq \norm{\la}{(X^k_{\rm{fin}})^*} \, \norm{F}{\Xkfin} \\
& \leq C\, \norm{\la}{(X^k_{\rm{fin}})^*} 
     \, (1+r_{B}) \, \mu(B)^{1/2} \, \norm{F}{2}
\quant F \in \QBkperp,
\end{aligned}
\end{equation}
where we have used Lemma~\ref{l: estimate of Xkfin norm}. 
By the Riesz Representation Theorem, there exists $G_B$ in $
\QBkperp$
such that 
$$
\la(F)
= (F,G_B)
\quant F \in \QBkperp,
$$
where $(\cdot,\cdot)$ denotes the inner product in $\QBkperp$,
i.e., the restriction to $\QBkperp$ of the 
inner product in $\ld{B}$.  Furthermore $\norm{\la_{\vert 
\QBkperp}}{}
= \norm{G_B}{2}$.   By combining this and (\ref{f: estimate la}),
we obtain that 
$$
\norm{G_B}{2}
\leq C\, \norm{\la}{(X^k_{\rm{fin}})^*} \, (1+r_B) \, \mu(B)^{1/2},
$$
where $C$ is independent of $B$.  
Taking the supremum over all balls $B$ in $\cB_{s_0}$ we obtain 
\begin{equation}\label{norms}
\sup_{B\in\cB_{s_0}}\Big(\frac{1}{\mu(B)}\,\
    \int_B|G_B|^2\wrt\mu\Big)^{1/2}
\leq C\, \norm{\la}{(X^k_{\rm{fin}})^*}. 
\end{equation}
Suppose that $B,B'$ are balls such that $B\subset B'$ and identify $L^2(B)$ with the subspace of all functions in $L^2(B')$ that vanish on $B'\setminus  B$.
Then $\QBkperp \subset q_k^2(B')^\perp$ and
$$
\int_B F \, G_B \wrt \mu
= \int_{B'} F \, G_{B'} \wrt \mu
\quant F \in \QBkperp.
$$
Hence $G_B = \PBk (G_{B'})$.   
As a consequence, $\bG := \{G_B: B\in\cB\}$ is in $\bYk(M)$, and 
$$
\norm{\bG}{\bYk}
\leq C\, \norm{\la}{(X^k_{\rm{fin}})^*}.  
$$
By Corollary~\ref{c: G and bG} there exists $G$ in $GBMO^k(M)$
such that $G_B = \pi_{k,B}(G)$ and $\norm{\bG}{\bYk} = \norm{G}
{GBMO^k}$.
Therefore $\la$ agrees with $\la_G$. 
\end{proof}

A corollary of the theory we developed is the following.

\begin{corollary} \label{c: equivalent norms}
If $k$ is a positive integer and $M$ has 
$C^{2k-2}$ bounded geometry  then the following hold:
\begin{enumerate}
\item[\itemno1]
the $\Xkfin$-norm and the $X^k$-norm are
equivalent on $\Xhfin{M}$;
\item[\itemno2]
suppose that $Z$ is a Banach space and that 
$\cT$ is a linear operator from $\Xhfin{M}$ into $Z$, such that
$$
L:= \sup \{ \norm{\cT A}{Z}: \hbox{$A$ {\rm{admissible}} $X^k$-
{\rm{atom}}} \}
< \infty.
$$
Then $\cT$ extends to a unique bounded linear operator from $
\Xh{M}$ to $Z$. 
\end{enumerate}
\end{corollary}

\begin{proof}
Part \rmi\ follows directly from the fact that $\Xhfin{M}$
and $\Xh{M}$ have isomorphic dual spaces.

To prove \rmii\ observe that a direct consequence of the 
assumption is that if
$F = \sum_{j=1}^N c_j \, A_j$ is a function in $\Xhfin{M}$,
then, by the triangle inequality,   
$$  
\norm{\cT F}{Z}
\leq L\, \sum_{j=1}^N \mod{c_j}\,.
$$ 
By taking the infimum over all representations of $F$
as a finite linear combination of $X^k$-atoms, we obtain
$$
\norm{\cT F}{Z}
\leq L\, \norm{F}{\Xkfin}
\leq C \, \norm{F}{\Xk}
\quant F \in \Xhfin{M}.
$$
We have used \rmi\ in the second inequality above.  
The required conclusion follows from the density of $\Xhfin{M}$ in 
$\Xh{M}$.
\end{proof}

Quite often one encounters the following situation.
Suppose that $\cT$ is a bounded linear operator on $\ld{M}$. 
Then $\cT$ is automatically defined on $\Xhfin{M}$. Assume that
$$
L : = \sup\{ \norm{\cT A}{L^1}: A~{\rm{admissible}}~ X^k-
{\rm{atom}} \} 
< \infty.
$$
By the previous results, the restriction of $\cT$ to $\Xhfin{M}$
has a unique extension to a bounded linear operator $\wt \cT$
from $\Xh{M}$ to $\lu{M}$. The question is whether the operators
$\cT$ and $\wt \cT$ are consistent, i.e., whether they coincide on 
the 
intersection $\Xh{M} \cap \ld{M}$ of their domains.

\begin{proposition}\label{p: consH1capL2}
Suppose that $k$ is a positive integer and $M$ has 
$C^{2k-2}$ bounded geometry
and that $\cT$ is bounded on $\ld{M}$.  The following hold:
\begin{enumerate}
\item[\itemno1]
if
$
L_0 : = \sup\{ \norm{\cT a}{L^1}: a~{\rm{admissible~}}H^1-
{\rm{atom}} \} 
< \infty,
$
then the unique continuous linear extension $\wt \cT$   
of the restriction of $\cT$ to $\hufin{M}$
to an operator from $\hu{M}$ to $\lu{M}$ 
agrees with $\cT$ on $\hu{M} \cap \ld{M}$;
\item[\itemno2]
if
$
L : = \sup\{ \norm{\cT A}{L^1}: A~{\rm{admissible~}} X^k-
{\rm{atom}} \} 
< \infty,   
$
then 
$$
\sup\{ \norm{\cT \cU^k a}{L^1}: a~{\rm{admissible~}} H^1-
{\rm{atom}} \} 
< \infty;   
$$
\item[\itemno3]
if
$
L : = \sup\{ \norm{\cT A}{L^1}: \hbox{\rm{$A$ admissible $X^k$-
atom}} \} 
< \infty,  
$
then the unique continuous linear extension $\wt \cT$   
of the restriction of $\cT$ to $\Xhfin{M}$ to an operator from $
\Xh{M}$ to $\lu{M}$
agrees with $\cT$ on $\Xh{M} \cap \ld{M}$.
\end{enumerate}
\end{proposition}

\begin{proof}
The proof of \rmi\
follows the same line of the proof of \cite[Proposition 4.2]{MSV1}, 
and is omitted. 

We give the proof of \rmii\ for $k=1$.   The proof in the case where 
$k\geq 2$
is similar and is omitted. 

Suppose that the atom $a$ is supported in the ball 
$B(p,R)$ with $R\leq s_0$. 
The proof hinges on the decomposition 
$$
\cU a 
= \sum_{i=0}^{d} A_i'
+  \sum_{j=1}^\infty \, A_j'',
$$ 
given in Lemma \ref{l: Uka}.
The function $4^i\,A_i'$ is a multiple of an admissible $X^1$-atom. 
Then 
$$
\bignorm{\cT(4^i\,A_i')}{L^1}\leq C\,L\,.
$$
Thus,
\begin{equation}\label{sumfin}
\Bignorm{\sum_{i=1}^d \cT A_i'}{L^1}\leq {C}\,L\,\sum_{i=1}^d4^{-i}
\leq C\,L, 
\end{equation}
where $C$ is independent of $a$. 

For every $j$ in $\BN$ by \eqref{ecodec} we have 
$$
\norm{\cT A_j''}{L^1}\leq C\,\e^{-\vep j}\,L.
$$ 
Thus, 
\begin{equation}\label{suminfin}
\sum_{j=1}^{\infty}\norm{\cT A_j''}{L^1} 
\leq C\,L\,\sum_{j=1}^{\infty}\e^{-\vep j} 
\leq C\,L.
\end{equation}
The inequalities \eqref{sumfin} and \eqref{suminfin} imply that
$$
\norm{\cT \cU a}{L^1}\leq C\,L,
$$
as required.

Finally, we prove \rmiii.
We consider the operator $\cT\circ \cU^k$, which is bounded on 
$\ld{M}$ and uniformly bounded on admissible $H^1$-atoms by 
\rmii. 
By \rmi\
the unique extension $\wt{\cT \circ \cU^k}$ of the restriction of 
$\cT\circ\cU^k$ to $\hufin{M}$ to an operator bounded from 
$\hu{M}$ to $\lu{M}$ agrees with $\cT\circ\cU^k$ on $\hu{M}\cap 
\ld{M}$. 

Then the operator $\big(\wt{\cT \circ \cU^k}\big)\circ \cU^{-k}$ 
is a bounded operator from $\Xh{M}$ to $\lu{M}$ 
which extends the restriction of $\cT$ to $\Xhfin{M}$. 
Then it coincides with the unique continuous linear extension $\wt 
\cT$ 
of the restriction of $\cT$ to $\Xhfin{M}$, 
i.e., $\wt \cT= \big(\widetilde{\cT \circ \cU^k}\big)\circ \cU^{-k}$. 

Moreover, for every function $F$ in $\Xh{M}\cap\ld{M}$ we have 
that 
$\cU^{-k}F$ is in $\hu{M}\cap\ld{M}$. Then
$$
\wt \cT F=\big(\wt{\cT \circ \cU^k}\big)\circ \cU^{-k}F
= \big(\cT \circ \cU^k\big) \circ \cU^{-k}F
= \cT F.
$$
Hence $\wt \cT$ agrees with $\cT$ on $\Xh{M}\cap\ld{M}$.
\end{proof}

\end{document}